\documentclass[11pt]{amsart}
\usepackage{amsmath}
\usepackage[active]{srcltx}
\usepackage[T1]{fontenc}
\usepackage[utf8]{inputenc}
\usepackage{lmodern}
\usepackage{microtype}
\usepackage{cite}
\usepackage{verbatim}
\usepackage{amsmath,amsfonts,amssymb,amsthm}
\usepackage[mathcal]{eucal}
\usepackage{enumerate}
\usepackage[centertags]{amsmath}
\usepackage{graphics}
\usepackage{refcount}
\usepackage{tikz}
\usepackage{xcolor}
\usepackage{pgfplots}
\usepackage{amsfonts}
\usepackage{dsfont}
\usepackage{enumitem}

\newcommand{\func}[1]{\operatorname{#1}}
\pgfplotsset{compat=1.18}
\newtheorem{theorem}{Theorem}
\newtheorem{prop}{Proposition}

\newtheorem{lemma}{Lemma}
\newtheorem{remark}{Remark}
\newtheorem{definition}{Definition}

\begin{document}
\author{Ushangi Goginava and Farrukh Mukhamedov }
\title[Weighted Walsh--Carleson Operators]{Dyadic Martingale Transforms and Weighted Walsh-Carleson Operators}
\address{U. Goginava, Department of Mathematical Sciences \\
United Arab Emirates University, P.O. Box No. 15551\\
Al Ain, Abu Dhabi, UAE}
\email{zazagoginava@gmail.com; ugoginava@uaeu.ac.ae}
\address{F. Mukhamedov, Department of Mathematical Sciences \\
United Arab Emirates University, P.O. Box No. 15551,\\
Al Ain, Abu Dhabi, UAE}
\email{farrukh.m@uaeu.ac.ae}
\subjclass[2020]{Primary 42C10; Secondary 42B25, 42B20, 60G42}
\keywords{Walsh system; weighted Walsh--Carleson operator; Walsh--Fourier series; martingale transforms; dyadic martingales; maximal operators; weak type $(1,1)$; summability and divergence}

\begin{abstract}
We study weighted Walsh--Carleson maximal operators arising from dyadic
martingale transforms associated with Walsh--Fourier partial sums. For
weights satisfying a uniform dyadic variation condition and a uniform
bound at the top dyadic scale, we prove weak type~$(1,1)$ estimates for the
corresponding maximal operators along subsequences. We also give divergence
criteria in terms of the behavior of the weights near the top dyadic scale
and, under suitable admissibility assumptions, relate these criteria to
explicit ratio conditions. As applications, we obtain results on matrix
transforms of Walsh--Fourier partial sums, including de la Vallée Poussin
means, Cesàro means with varying parameters, Nörlund logarithmic means, and
general Nörlund means. In particular, we prove a Walsh--Paley analogue of
the Leindler--Tandori theorem and establish everywhere divergence results
for several summability methods.
\end{abstract}

\maketitle

\section{Introduction}

It is well known that almost everywhere convergence of sequences of
operators is intimately related to the weak-type behavior of the
corresponding maximal operators (see, for example, Stein~\cite{Stein}). In
the context of Fourier analysis, establishing weak-type inequalities for
such maximal operators is typically a delicate matter, largely due to their
singular nature. A prototypical example is provided by the Carleson--Hunt
theorem, which asserts the almost everywhere convergence of trigonometric
Fourier series for functions in $L_{p}([-\pi,\pi])$ with $p>1$, and whose
proof fundamentally relies on a weak-type estimate for the Carleson maximal
operator 
\begin{equation}
A^{\ast}(f)(x) := \sup_{n\in \mathbb{Z}} \left\vert \int_{-\pi}^{\pi}
f(x-t)\,\frac{e^{int}}{t}\,dt \right\vert.  \label{CO}
\end{equation}

In recent years there has been substantial progress in the development of
methods for estimating the operator~\eqref{CO} and various generalizations
thereof. Important contributions include the works of Sjölin~\cite{Sj1971},
Antonov~\cite{Ant1996}, Sjölin and Soria~\cite{SF2003}, Reyna~\cite%
{Arias-2002}, Hytönen, Lacey and Parissis~\cite{Lacey-2024}, Krause and
Lacey~\cite{Lacey2017,Lacey2018}, Oberlin, Seeger, Tao, Thiele and Wright~%
\cite{Tao2012}, Muscalu, Tao
and Thiele~\cite{Tao2006}, Benea, Bernicot, Lie and Vitturi~\cite{LieAdv2024}%
, as well as Lie~\cite{LieAdv22024,LieAnal2020}. A common feature of these
works is the heavy use of combinatorial properties of dyadic intervals
together with the fine behavior of geometric maximal operators.

Another central theme in this direction is the construction of
counterexamples demonstrating the optimality of function spaces, typically
Orlicz classes $\varphi(L)$, for various approximation processes. The
methods introduced by Kolmogorov~\cite{kolmogoroff1923serie}, Fefferman~\cite%
{Feff1971}, Stein~\cite{Stein}, Konyagin~\cite{Kon-2000}, Bochkarev~\cite%
{BochDiv,Bochk}, Olevskii~\cite{Ol1,Ol2}, Schipp et al.~\cite{SWS}, G\'at~%
\cite{gat2024}, Gát, Goginava and Karagulyan~\cite{GGKJMAA2015},
Karagulyan~\cite{karagula2016}, Getsadze~\cite{Gets1,Gets2}, Pan and Ai~\cite%
{Pan2023}, Oniani~\cite{On1,On2}, Goginava and Oniani~\cite{GoOn2021},
Goginava~\cite{GogJMAA2024}, and Goginava and Mukhamedov~\cite{GMProc2025}
are in a sense universal and often admit flexible modifications. These
techniques have proved to be powerful tools for constructing sharp
counterexamples in a variety of settings.

In the framework of Walsh systems, questions of almost everywhere
convergence and divergence are closely connected with martingale
transformations.
The partial sums of the Walsh--Fourier series of an integrable $f
$ can be written in the form 
\begin{equation*}
S_{n}(f)=w_{n}\sum_{k=0}^{\infty }\varepsilon _{k}(n)\,\mathcal{E} _{k}(fw_{n}),
\end{equation*}%
where $E_{k}$ denotes conditional expectation with respect to the $k$-th
dyadic $\sigma$-algebra and
\begin{equation*}
\mathcal{E}_{k}f:=E_{k+1}f-E_{k}f
\end{equation*}%
is the corresponding dyadic martingale difference operator. In this formula,
$\varepsilon _{k}(n)$ denotes the $k$-th Walsh digit of $n$, and $w_{n}$ is
the $n$-th Walsh function. Thus $S_{n}(f)$ can be represented as a martingale
transform.

Let $\{\Omega _{k}(n):k,n\in \mathbb{N}\}$ be a positive family such that,
for each fixed $n\in \mathbb{N}$, the sequence $\{\Omega _{k}(n)\}_{k\geq 0}$
is nondecreasing in $k$, that is, 
\begin{equation*}
0<\Omega _{k-1}(n)\leq \Omega _{k}(n),\qquad k,n\in \mathbb{N}.
\end{equation*}%
We introduce the operator 
\begin{equation*}
M_{n}(\boldsymbol{\Omega })f:=\sum_{k=0}^{\infty }\varepsilon
_{k}(n)\,\Omega _{k}(n)\,\mathcal{E}_{k}(fw_{n}),
\end{equation*}
which represents a martingale transform. Then we have
\begin{equation}
S_{n}(f)=w_{n}\,M_{n}(\boldsymbol{\Omega })f,  \label{s}
\end{equation}%
where $\Omega _{k}(n)\equiv 1$ for all $k\in \mathbb{N}$.  Equality~\eqref{s} emphasizes the close relationship between martingale
transforms and the partial sums of Walsh--Fourier series. Detailed accounts
of martingale transforms and their connections with Walsh dyadic analysis can
be found in \cite{WeBook2,Weisz-book1,BurkMT,chao1992,frangos1987,imk88,fraimk88}.

Based on this definition, we consider the corresponding weighted
Walsh--Carleson maximal operator 
\begin{equation}
W_{C}(\boldsymbol{\Omega})f := \sup_{n\in \mathbb{N}} \left\vert
\sum_{k=0}^{\infty} \varepsilon_{k}(n)\,\Omega_{k}(n)\,\mathcal{E}_{k}(f w_{n})
\right\vert .  \label{mt}
\end{equation}
We also introduce the weighted Walsh--Carleson maximal operator along a
subsequence $\{n_{a}: a\in \mathbb{N}\}$, defined by 
\begin{equation*}
W_{C}(\boldsymbol{\Omega},\{n_{a}\})f := \sup_{a\in \mathbb{N}} \left\vert
\sum_{k=0}^{\infty} \varepsilon_{k}(n_{a})\,\Omega_{k}(n_{a})\,\mathcal{E}_{k}(f
w_{n_{a}}) \right\vert .
\end{equation*}

Next, we define the kernel 
\begin{equation*}
P_{n}(\boldsymbol{\Omega}) := \sum_{k=0}^{\infty}
\varepsilon_{k}(n)\,\Omega_{k}(n)\,r_{k} D_{2^{k}},
\end{equation*}
where $r_{k}$ is the $k$-th Rademacher function and $D_{2^{k}}$ is the
Walsh--Dirichlet kernel of order $2^{k}$. It is straightforward to verify
that $M_{n}(\boldsymbol{\Omega})$ admits the convolution representation 
\begin{equation}
M_{n}(\boldsymbol{\Omega})f = (f w_{n}) \ast P_{n}(\boldsymbol{\Omega}) =
w_{n}\bigl( f \ast (w_{n} P_{n}(\boldsymbol{\Omega})) \bigr),  \label{conv}
\end{equation}
where $\ast$ denotes convolution on the dyadic group.

We observe that, for the classical Walsh--Carleson maximal operator, one takes
\begin{equation}
\Omega_{k}(n)=1,
\qquad 0\le k\le |n|,  \label{ps}
\end{equation}
with the convention that the terms with $k>|n|$ do not contribute because
$\varepsilon_k(n)=0$. With this choice, the operator in~\eqref{mt} reduces to
the classical Walsh--Carleson maximal operator.

Thus, the operator \eqref{mt} can be regarded as a weighted generalization
of the Walsh--Carleson operator. As will be shown below, it is also closely
connected with classical summability problems for Walsh--Fourier series.

\begin{definition}
\label{dee} Let $\bigl(\Omega _{k}(n)\bigr)_{k,n\geq 1}$ be a family of
weights, and let $\boldsymbol{\gamma }=\bigl(\gamma (n)\bigr)_{n=1}^{\infty }
$ be a sequence of positive integers. We say that $\gamma $ is a \emph{%
sequence of divergence with respect to} $\bigl(\Omega _{k}(n)\bigr)_{k,n\geq
1}$ if the following conditions hold:

\begin{enumerate}
\item $\gamma : \mathbb{N} \to \mathbb{N}$;

\item the sequence tends to infinity uniformly with respect to the dyadic
order, that is,
\begin{equation*}
\lim_{N\to\infty}\inf\{\gamma(n): n\in\mathbb{P},\ |n|\ge N\}=\infty;
\end{equation*}

\item 
\begin{equation*}
\sup_{n\in \mathbb{P}}\frac{\gamma (n)}{|n|}<1;
\end{equation*}

\item there exists a constant $c>0$ such that 
\begin{equation}
\Omega _{|n|-\gamma (n)}(n)\,\geq \,c\quad \text{for all }n\in \mathbb{P}.
\label{LME}
\end{equation}
\end{enumerate}
\end{definition}

The purpose of the present paper is to obtain sharp conditions on the weight
system $\boldsymbol{\Omega }$ under which the maximal operator $W_{C}(%
\boldsymbol{\Omega })$ is of weak type~$(1,1)$. Our main result is the
following theorem.

\begin{theorem}
\label{conj1} Let $\boldsymbol{\Omega }=\bigl(\Omega _{k}(n)\bigr)_{k,n\in 
\mathbb{N}}$ be a system of weights, and let $\left( n_{a}\right) _{a\in 
\mathbb{N}}\subset \mathbb{P}$ be a sequence of indices.

\medskip \noindent \textup{(a) Boundedness.} Assume that the weights along
the sequence $\left( n_{a}\right) _{a\in \mathbb{N}}$ satisfy the uniform
summability condition 
\begin{equation}
\sup_{a\in \mathbb{N}}\sum_{k=0}^{|n_{a}|}\bigl|\varepsilon
_{k-1}(n_{a})-\varepsilon _{k}(n_{a})\bigr|\,\Omega _{k}(n_{a})<\infty .
\label{cond}
\end{equation}%
Assume also the top-scale bound along the same sequence,
\begin{equation}
\sup_{a\in\mathbb{N}}\Omega_{|n_a|}(n_a)<\infty .
\label{top-bound-seq}
\end{equation}
Then the maximal operator $W_{C}(\boldsymbol{\Omega },\{n_{a}\})$ is of weak
type~\textup{(1,1)}; that is, equivalently, 
\begin{equation}
\Vert W_{C}(\boldsymbol{\Omega },\{n_{a}\})\Vert _{L_{1}(\mathbb{I}%
)\rightarrow L_{1,\infty }(\mathbb{I})}<\infty \label{weeek} .
\end{equation}

\medskip \noindent \textup{(b) Divergence.} If there exists a sequence of
divergence with respect to $\bigl(\Omega _{k}(n)\bigr)_{n=1}^{\infty }$,
then the full maximal operator $W_{C}(\boldsymbol{\Omega })$ is not of weak
type~\textup{(1,1)}; that is, 
\begin{equation*}
\Vert W_{C}(\boldsymbol{\Omega })\Vert _{L_{1}(\mathbb{I})\rightarrow
L_{1,\infty }(\mathbb{I})}=\infty .
\end{equation*}
\end{theorem}

\begin{remark}
	In part (b) of Theorem \ref{conj1} the problem can be formulated in a more
	general setting. Let $\left\{ n_{a}:a\in \mathbb{N}\right\} $ be a
	subsequence of integers, and suppose that condition (\ref{cond}) fails along
	the sequence $\left\{ n_{a}:a\in \mathbb{N}\right\} $ . The question is
	whether this failure necessarily implies that 
	\begin{equation*}
		\Vert W_{C}(\mathbf{\Omega ,}\left\{ n_{a}\right\} )f\Vert _{L_{1}(\mathbb{I}%
			)\rightarrow L_{1,\infty }(\mathbb{I})}=\infty .
	\end{equation*}%
	In general, this problem has a negative answer. Indeed, in his solution of
	the Balashov problem, Konyagin \cite{kon} proved that, in the case of partial sums (see
	(\ref{ps})), there exists a sequence $\left\{ n_{a}:a\in \mathbb{N}\right\} $
	for which condition (\ref{cond}) does not hold, while condition (\ref{weeek}%
	) is nevertheless satisfied. \newline
	\newline
	\textbf{\textit{In the general setting, however, this problem remains open.}}
	Moreover, even in the special case of partial sums (\ref{ps}), the problem
	of finding necessary and sufficient conditions on a subsequence $\left\{
	n_{a}:a\in \mathbb{N}\right\} $ that guarantee condition (\ref{weeek}) is
	still unresolved. This problem was posed by Konyagin \cite{Kon2006} in 2006,
	at the \textbf{International Congress of Mathematicians}, and a partial
	solution was obtained in a paper by the first author jointly with Oniani 
	\cite{GoOn2021}.
\end{remark}

Since the divergence part of Theorem~\ref{conj1} is formulated in terms of
the existence of a sequence of divergence with respect to $\bigl(\Omega_k(n)%
\bigr)_{n=1}^{\infty}$, this criterion is not, at first sight, easy to
verify in concrete applications. This naturally leads to the following two
questions:

\begin{enumerate}
\item Can one provide simple, verifiable criteria for the existence of a
sequence of divergence?

\item What is the precise relationship between the existence of a sequence
of divergence and the stronger condition 
\begin{equation}
\sup_{n\in \mathbb{N}}\sum_{k=0}^{|n|}\Omega _{k}(n)=\infty \,?
\label{Omega-sum-infty}
\end{equation}
\end{enumerate}

First we observe that, in general, condition \eqref{Omega-sum-infty} does
not guarantee the existence of a sequence of divergence. Indeed, consider
the family of weights 
\begin{equation*}
\Omega_k(n) := \frac{1}{|n|-k+1}, \qquad n \in \mathbb{N},\ 0 \le k \le |n|.
\end{equation*}
For each fixed $n$ we have 
\begin{equation*}
\sum_{k=0}^{|n|} \Omega_k(n) = \sum_{k=0}^{|n|} \frac{1}{|n|-k+1} =
\sum_{j=1}^{|n|+1} \frac{1}{j},
\end{equation*}
so that 
\begin{equation*}
\sup_{n\in\mathbb{N}} \sum_{k=0}^{|n|} \Omega_k(n) = \sup_{m\in\mathbb{N}}
\sum_{j=1}^m \frac{1}{j} = \infty,
\end{equation*}
and hence condition \eqref{Omega-sum-infty} is satisfied.

On the other hand, for any integer $\gamma $ with $0\leq \gamma \leq |n|$ we
have 
\begin{equation*}
\Omega _{|n|-\gamma }(n)=\frac{1}{|n|-(|n|-\gamma )+1}=\frac{1}{\gamma +1}.
\end{equation*}%
In particular, for any choice of integers $\gamma (n)\rightarrow \infty $
with $0\leq \gamma (n)<|n|$, we obtain 
\begin{equation*}
\Omega _{|n|-\gamma (n)}(n)=\frac{1}{\gamma (n)+1}\longrightarrow 0\quad 
\text{as }n\rightarrow \infty .
\end{equation*}%
Consequently, there is no sequence $\gamma :\mathbb{N}\rightarrow \mathbb{N}$
with $\gamma (n)\rightarrow \infty $ for which there exists a constant $c>0$
satisfying 
\begin{equation*}
\Omega _{|n|-\gamma (n)}(n)\geq c\quad \text{for all }n.
\end{equation*}

As we shall see below, the underlying obstruction in this example is the
behaviour of the double ratio limit 
\begin{equation}
\lim_{k,n\rightarrow \infty }\frac{\Omega _{k}(n)}{\Omega _{k-1}(n)},
\label{limmE2}
\end{equation}%
which, for the weights considered above, does not exist. Indeed, we have 
\begin{equation*}
\frac{\Omega _{k}(n)}{\Omega _{k-1}(n)}=1+\frac{1}{|n|-k+1}.
\end{equation*}%
If we take $k=|n|$, then $|n|-k=0$, and hence 
\begin{equation*}
\frac{\Omega _{|n|}(n)}{\Omega _{|n|-1}(n)}=2.
\end{equation*}%
In contrast, along any sequence $(k_{j},n_{j})$ in the triangular index
region $\{(k,n):1\leq k<|n|\}$ for which $|n_{j}|-k_{j}\rightarrow \infty $,
we have 
\begin{equation*}
\frac{\Omega _{k_{j}}(n_{j})}{\Omega _{k_{j}-1}(n_{j})}=1+\frac{1}{%
|n_{j}|-k_{j}+1}\longrightarrow 1.
\end{equation*}%
Thus the limit in (\ref{limmE2}) fails to exist. 

To this end, we first introduce the notion of the limit of a double sequence
over a cone. First, we define the following cone: there exists a constant $%
\kappa \in (0,1)$ such that%
\begin{equation*}
\Delta _{\kappa }:=\bigl\{(k,n)\in \mathbb{N}\times \mathbb{P}:\kappa
|n|<k\leq |n|\bigr\}.
\end{equation*}%
We now introduce a sharper cone defined as follows. Let $\omega =(\omega
_{n})_{n=0}^{\infty }$ be a sequence of positive integers satisfying the
following conditions: 
\begin{equation}
\omega _{n}\nearrow \infty \quad \text{and}\quad \frac{|n|}{\omega _{n}}%
\nearrow \infty .  \tag{M}
\end{equation}%
Define%
\begin{equation*}
\Delta _{\omega }:=\bigl\{(k,n)\in \mathbb{N}\times \mathbb{P}:|n|-\omega
_{n}<k\leq |n|\bigr\}.
\end{equation*}

We notice that $\Delta _{\omega }\subset \Delta _{\kappa}$. For the sake of shortness, we write $\Delta $ to denote either $\Delta_{\kappa }$ or $\Delta_{\omega }$. Given a family $\left( a_{k}(n)\right) _{(k,n)\in \Delta}$ and $L\in \mathbb{C}$, we say that 
\begin{equation*}
\lim_{(k,n)_{\Delta }\rightarrow \infty }a_{k}(n)=L
\end{equation*}%
if for every $\varepsilon >0$ there exists $N_{0}\in \mathbb{N}$ such that 
\begin{equation*}
|a_{k}(n)-L|<\varepsilon \quad \text{whenever }(k,n)\in \Delta \text{ and }%
|n|\geq N_{0}.
\end{equation*}

A sequence $\{\Omega _{k}(n)\}$ is called \textit{$\Delta$-admissible}, if it satisfies the following conditions:
\begin{itemize}
	\item[(a)]  there exists a finite limit $L$ such
	that 
	\begin{equation}
		\lim_{(k,n)_{\Delta }\rightarrow \infty }\frac{\Omega _{k}(n)}{\Omega
			_{k-1}(n)}=L.  \label{limmE}
	\end{equation}
	\item[(b)]  The sequence $\{\Omega _{|n|}(n)\}$ is uniformly bounded above and
	below, that is, there exist constants $0<c\leq C<\infty $ such that 
	\begin{equation}
		c\leq \Omega _{|n|}(n)\leq C\qquad (n\in \mathbb{P}).  \label{<1}
	\end{equation}%
	Equivalently, $\Omega _{|n|}(n)\asymp 1$.
\end{itemize} 

Now we answer Question~1 formulated above.

\begin{theorem}\label{T2}  Let $\{\Omega _{k}(n)\}$ be a $\Delta$-admissible sequence. Then, the following statements are equivalent:

\begin{enumerate}
\item[(i)] The triangular limit  satisfies 
\begin{equation}
\lim_{(k,n)_{\Delta }\rightarrow \infty }\frac{\Omega _{k}(n)}{\Omega
_{k-1}(n)}=1;  \label{sim}
\end{equation}

\item[(ii)] There exists a sequence of divergence with respect to $\bigl(%
\Omega_k(n)\bigr)_{k,n\ge 1}$.
\end{enumerate}
\end{theorem}

Next we address Question~2 formulated above and clarify the relationship
between $\boldsymbol{\gamma }$ is a \emph{sequence of divergence with
respect to} $\bigl(\Omega _{k}(n)\bigr)_{k,n\geq 1}$ and %
\eqref{Omega-sum-infty}.

\begin{theorem}
\label{T3}  Let $\kappa\in(0,1)$ be fixed, and assume that
$\{\Omega_k(n)\}$ is $\Delta_\kappa$-admissible. Then the following two
conditions are equivalent:

\begin{enumerate}
\item[(i)]
\begin{equation}  \label{kapa}
\lim_{(k,n)_{\Delta_\kappa}\to\infty} \frac{\Omega_k(n)}{\Omega_{k-1}(n)} =
1;
\end{equation}

\item[(ii)] One has
\begin{equation}  \label{div22}
\sup_{n\in\mathbb{N}} \sum_{k=0}^{|n|} \Omega_k(n) = \infty.
\end{equation}
\end{enumerate}

Moreover, if $\{\Omega_k(n)\}$ is $\Delta_\omega$-admissible and
\begin{equation}  \label{omega}
\lim_{(k,n)_{\Delta_\omega}\to\infty} \frac{\Omega_k(n)}{\Omega_{k-1}(n)} =
1,
\end{equation}
then \eqref{div22} follows. This sufficient condition is not necessary in
general.
\end{theorem}


\begin{remark}
The boundedness of maximal operators associated with partial sums is one of
the central topics in Fourier analysis. Theorems~\ref{conj1}, \ref{T2} and %
\ref{T3} show that, in the Walsh setting, the weak-type $(1,1)$ boundedness
of the weighted Walsh--Carleson maximal operator reduces to the simple and
explicit condition~\textup{(\ref{sim})} on the weight $\mathbf{\Omega }$.
\end{remark}

We next discuss everywhere divergence and almost everywhere convergence for sequences of operators associated with the matrix 
	$\mathbb{T}:=(t_{k,n})_{k,n\in \mathbb{P}}$ which enjoys 
the following assumptions:
\begin{enumerate}
	\item $t_{k,n} \ge 0$ for all $k,n\in\mathbb{P}$;
	
	\item $t_{k,n} \le t_{k+1,n}$ for all $k,n\in\mathbb{P}$ (rowwise
	monotonicity);
	
	\item $\displaystyle\sum_{k=0}^{n}t_{k,n}=1$ for all $n\in \mathbb{P}$.
\end{enumerate}
	
By means of $\mathbb{T}$, the $n$th triangular matrix transform of the Walsh--Fourier series of $f$ is
	defined by 
	\begin{equation}
		\mathcal{T}_{n}^{\mathbb{T}}(f;x) := \sum_{k=1}^{n} t_{k,n}\,S_{k}(f;x),
		\qquad n\in\mathbb{P}.  \label{weight-mean}
	\end{equation}

\begin{theorem}
\label{divergence}  The following statements hold: 

\begin{enumerate}
\item[\textup{(1)}] \textbf{Uniform boundedness implies pointwise
convergence.}

Let $\{n_{a}: a\in \mathbb{N}\}$ be any increasing sequence of natural
numbers.  If  
\begin{equation}  \label{h2}
\sup_{a\in \mathbb{N}} \bigl\|\mathcal{T}_{n_{a}}^{\mathbb{T}}\bigr\|_{H_{1}(%
\mathbb{I}) \to L_{1}(\mathbb{I})} < \infty,
\end{equation}
then for every $f\in L_{1}(\mathbb{I})$ one has  
\begin{equation*}
\mathcal{T}_{n_{a}}^{\mathbb{T}}(f;x) \longrightarrow f(x)  \quad\text{as }
a\to\infty,  
\end{equation*}
for almost every $x\in \mathbb{I}$.

\item[\textup{(2)}] \textbf{Everywhere divergence.}

If  
\begin{equation}  \label{eqq2}
\lim_{(k,n)_{\Delta_{\omega}}\to\infty} \frac{\widetilde{T}_{2^{k},n}}{%
\widetilde{T}_{2^{k-1},n}} = 1,
\end{equation}
then there exists a function $f_{0}\in L_{1}(\mathbb{I})$ such that  
\begin{equation*}
\limsup_{n\to\infty}  \left|\mathcal{T}_{n}^{\mathbb{T}}(f_{0};x)\right|  =
\infty  
\end{equation*}
at every point $x\in \mathbb{I}$. 
\end{enumerate}
\end{theorem}

To establish this result, we use the kernel decomposition recalled in
Appendix~A and apply Theorem~\ref{conj1} to the weights
\begin{equation*}
\Omega_s(n)=\widetilde T_{2^s,n},\qquad 0\le s\le |n|.
\end{equation*}

\begin{remark}
If, in addition, the weights satisfy  
\begin{equation*}
\lim_{(k,n)_{\Delta_{\kappa}}\to\infty}  \frac{\widetilde{T}_{2^{k},n}}{%
\widetilde{T}_{2^{k-1},n}} = 1  
\end{equation*}
for some $\kappa\in(0,1)$, then, by Theorem~\ref{T2}, condition  \eqref{eqq2}
in Theorem~\ref{divergence} may be equivalently replaced by  
\begin{equation*}
\sup_{n\in \mathbb{N}}  \bigl\|\mathcal{T}_{n}^{\mathbb{T}}\bigr\|_{H_{1}(%
\mathbb{I})\to L_{1}(\mathbb{I})}  = \infty.  
\end{equation*}
\end{remark}

\begin{remark}
Part~\textup{(1)} of Theorem~\ref{divergence} was established in  \cite%
{GogMukh2024}. The proof of part~\textup{(2)} relies on several auxiliary 
lemmas and is given below.
\end{remark}

Theorem~\ref{divergence} can be applied to recover many classical results,
as well as to obtain new ones, for the Walsh--Paley system.

\subsection{Divergence of partial sums of Walsh-Fourier series}

Let%
\begin{equation*}
t_{k,n}:=\mathds{1}_{\{n\}}(k),n\in \mathbb{N}.
\end{equation*}%
Then the matrix $\mathbb{S}:=\{t_{k,n}:k,n\in \mathbb{N}\}$ generates the
sequence of partial sum operators $\mathcal{T}_{n}^{\mathbb{S}}$, that is, 
\begin{equation*}
\mathcal{T}_{n}^{\mathbb{S}}(f;x)=S_{n}(f;x).
\end{equation*}

In this case it is easy to see that%
\begin{equation*}
\lim_{(k,n)_{\Delta _{\kappa }}\rightarrow \infty }\left( \frac{\widetilde{T}%
_{2^{k},n}}{\widetilde{T}_{2^{k-1},n}}\right) =1.
\end{equation*}%
Therefore, by Theorem~\ref{conj1}, 
\begin{equation}
\Bigl\|\sup_{n\in \mathbb{N}}|\mathcal{T}_{n}^{\mathbb{S}}|\Bigr\|_{L_{1}(%
\mathbb{I})\rightarrow L_{1,\infty }(\mathbb{I})}=\infty .  \label{eq-st}
\end{equation}

On the other hand, invoking Theorem~\ref{divergence} in this particular
situation, we recover the classical result of Stein~\cite{Stein}, which
asserts that there exists an integrable function $f\in L_{1}(\mathbb{I})$
whose Walsh--Fourier series has divergent partial sums at every point $x\in 
\mathbb{I}$.

\subsection{Almost everywhere convergence of the de la Vallée Poussin means
of Walsh--Fourier series}

Let $\{\lambda_{n}\}_{n=1}^{\infty}$ be a nondecreasing sequence of integers satisfying $1\leq \lambda_n\leq n$. We define the de la Vallée Poussin means by 
\begin{equation}
V_{n}^{(\lambda)}(f;x) := \frac{1}{\lambda_{n}+1} \sum_{k=n-\lambda_{n}}^{n}
S_{k}(f;x),  \label{VP}
\end{equation}
where $S_{k}(f;x)$ denotes the $k$-th partial sum of the Fourier series of $f
$. Leindler~\cite{lein68} proved that if 
\begin{equation}
n = O(\lambda_{n}) \qquad (n\to\infty),  \label{leid}
\end{equation}
then $V_{n}^{(\lambda)}(f;x) \to f(x)$ as $n\to\infty$ almost everywhere for
every integrable function $f$. He also raised the question whether condition %
\eqref{leid} is essential. Tandori~\cite{tand1979} solved Leindler's problem
and showed that condition~\eqref{leid} is indeed necessary.

The analogous problem for the Walsh--Paley system has remained open. Theorem~%
\ref{divergence} provides the Walsh--Paley counterpart of the
Leindler--Tandori theorem and yields a complete solution of this problem.

\begin{theorem}
\label{VP-W}  \textbf{Leindler--Tandori theorem for the Walsh--Paley system} 

\begin{enumerate}
\item[\textup{(a)}]  If condition~\eqref{leid} holds, then for every
integrable function  $f \in L_{1}(\mathbb{I})$ the de la Vallée Poussin
means \eqref{VP} with  respect to the Walsh--Paley system converge almost
everywhere to $f$, i.e.  
\begin{equation*}
V_{n}^{(\lambda)}(f;x) \longrightarrow f(x)  \quad \text{as } n\to\infty,  
\end{equation*}
for almost every $x \in \mathbb{I}$.

\item[\textup{(b)}]  If  
\begin{equation}
\frac{n}{\lambda_{n}} \longrightarrow \infty \qquad (n\to\infty),
\label{div55}
\end{equation}
then there exists an integrable function $f \in L_{1}(\mathbb{I})$ for which
the sequence $\{V_{n}^{(\lambda)}(f;x)\}_{n=1}^{\infty}$ diverges everywhere
on $\mathbb{I}$. 
\end{enumerate}
\end{theorem}

We also note that significant results on divergence phenomena for sequences
obtained by matrix transformations of partial sums of the trigonometric
system were obtained by Totik~\cite{totik1982}. We further emphasize the
recent paper of G\'at~\cite{gat2024}, where he answered a question posed by
Zalcwasser~\cite{Zalcwasser1936} in 1936 concerning almost everywhere
convergence of arithmetic means of subsequences of partial sums of Fourier
series.

\subsection{Almost everywhere convergence of the Ces\`aro means of
Walsh-Fourier series}

Let 
\begin{equation*}
\mathbb{A}:=\left\{ t_{k,n}:=\frac{A_{n-k}^{\alpha _{n}-1}}{A_{n}^{\alpha
_{n}}}:k=0,\dots ,n,\ n\in \mathbb{N}\right\} ,
\end{equation*}%
where $\{\alpha _{n}:n\in \mathbb{N}\}\subset (0,1]$ and 
\begin{equation*}
A_{n}^{\alpha _{n}}:=\frac{(\alpha _{n}+1)\cdots (\alpha _{n}+n)}{n!},\qquad
A_{0}^{\alpha _{n}}:=1.
\end{equation*}%
We have%
\begin{equation*}
\frac{\widetilde{T}_{2^{s},n}}{\widetilde{T}_{2^{s-1},n}}=\left( 1+\frac{%
\alpha _{n}}{2^{s-1}}\right) \cdots \left( 1+\frac{\alpha _{n}}{2^{s}-1}%
\right) .
\end{equation*}%
Hence%
\begin{eqnarray*}
\ln \left( \frac{\widetilde{T}_{2^{s},n}}{\widetilde{T}_{2^{s-1},n}}\right)
&=&\sum\limits_{k=2^{s-1}}^{2^{s}-1}\ln \left( 1+\frac{\alpha _{n}}{k}%
\right) =\sum\limits_{k=2^{s-1}}^{2^{s}-1}\left( \frac{\alpha _{n}}{k}%
+O\left( \frac{1}{k^{2}}\right) \right) \\
&=&\alpha _{n}\sum\limits_{k=2^{s-1}}^{2^{s}-1}\frac{1}{k}+O\left( \frac{1}{%
2^{s}}\right) =\alpha _{n}\ln \left( 2\right) +o\left( 1\right) \text{ \ }%
\left( s\rightarrow \infty \right) .
\end{eqnarray*}

We consider two cases:

\noindent \textbf{Case (a):} $\alpha _{n}\equiv \alpha \in (0,1]$ is
constant. In this situation one has 
\begin{equation*}
\lim_{(s,n)_{\Delta _{\kappa }}\rightarrow \infty }\left( \frac{\widetilde{T}%
_{2^{s},n}}{\widetilde{T}_{2^{s-1},n}}\right) =2^{\alpha }>1.
\end{equation*}%
Consequently, by Theorem~\ref{conj1}, the maximal operator 
\begin{equation*}
f\longmapsto \sup_{n\in \mathbb{N}}|\mathcal{T}_{n}^{\mathbb{A}}f|
\end{equation*}%
is of weak type~$(1,1)$; that is, 
\begin{equation*}
\Bigl\|\sup_{n\in \mathbb{N}}|\mathcal{T}_{n}^{\mathbb{A}}|\Bigr\|_{L_{1}(%
\mathbb{I})\rightarrow L_{1,\infty }(\mathbb{I})}<\infty .
\end{equation*}%
This result goes back to Schipp~\cite{schipp1975certain}.

\noindent \textbf{Case (b):} $\displaystyle\lim_{n\rightarrow \infty }\alpha
_{n}=0$. In this case we obtain 
\begin{equation*}
\lim_{(s,n)_{\Delta _{\kappa }}\rightarrow \infty }\left( \frac{\widetilde{T}%
_{2^{s},n}}{\widetilde{T}_{2^{s-1},n}}\right) =\lim_{(s,n)_{\Delta \left( 
\boldsymbol{\omega ,}\mathcal{M}\right) }\rightarrow \infty }2^{\alpha
_{n}}=1.
\end{equation*}%
Hence, by Theorem~\ref{conj1}, the maximal operator $\sup_{n\in \mathbb{N}}|%
\mathcal{T}_{n}^{\mathbb{A}}|$ fails to be of weak type~$(1,1)$, that is, 
\begin{equation*}
\Bigl\|\sup_{n\in \mathbb{N}}|\mathcal{T}_{n}^{\mathbb{A}}|\Bigr\|_{L_{1}(%
\mathbb{I})\rightarrow L_{1,\infty }(\mathbb{I})}=\infty .
\end{equation*}

Moreover, Theorem~\ref{divergence} yields the existence of an integrable
function $f\in L_{1}(\mathbb{I})$ such that the Ces\`{a}ro means $\mathcal{T}%
_{n}^{\mathbb{A}}(f;x)$ diverge at every point $x\in \mathbb{I}$. We note
that the almost-everywhere divergence in this setting was established
recently in~\cite{GogJMAA2024}. From the viewpoint of \emph{everywhere}
divergence, the result obtained here appears to be new. We also remark that
the analogous question for the trigonometric system remains open.

\subsection{Almost everywhere divergence of the Nörlund logarithmic means of
Walsh-Fourier series}

Consider the Nörlund logarithmic means generated by the matrix 
\begin{equation*}
\mathbb{L}:=\left\{ t_{k,n}:=\frac{1}{\ln (n)}\frac{1}{n-k}:k=0,\dots ,n-1,\
n\in \mathbb{N}\right\} ,
\end{equation*}%
where 
\begin{equation*}
\ln (n):=\sum_{k=1}^{n}\frac{1}{k}
\end{equation*}%
denotes the $n$th harmonic number. A straightforward calculation shows that 
\begin{equation*}
\lim_{(s,n)_{\Delta _{\kappa }}\rightarrow \infty }\left( \frac{\widetilde{T}%
_{2^{s},n}}{\widetilde{T}_{2^{s-1},n}}\right) =\lim_{(s,n)_{\Delta _{\kappa
}}\rightarrow \infty }\frac{s\ln \left( 2\right) +\gamma +2^{-s-1}+O\left(
2^{-2s}\right) }{\left( s-1\right) \ln \left( 2\right) +\gamma
+2^{-s}+O\left( 2^{-2s}\right) }=1.
\end{equation*}%
Hence the maximal operator 
\begin{equation*}
f\longmapsto \sup_{n\in \mathbb{N}}|\mathcal{T}_{n}^{\mathbb{L}}f|
\end{equation*}%
does \emph{not} satisfy a weak type~$(1,1)$ inequality; that is, 
\begin{equation*}
\Bigl\|\sup_{n\in \mathbb{N}}|\mathcal{T}_{n}^{\mathbb{L}}|\Bigr\|_{L_{1}(%
\mathbb{I})\rightarrow L_{1,\infty }(\mathbb{I})}=\infty .
\end{equation*}

Applying Theorem~\ref{divergence}, we deduce the existence of an integrable
function $f\in L_{1}(\mathbb{I})$ whose Nörlund logarithmic means (with
respect to the Walsh system) diverge everywhere. Almost-everywhere
divergence for $\mathcal{T}_{n}^{\mathbb{L}}$ was proved by the first author
together with Gát in~\cite{GatGogiAMSDiv2}. More recently, this phenomenon
was extended to general orthogonal systems, yielding in particular a Nö%
rlund-logarithmic analogue of Bochkarev's classical theorem in~\cite%
{GMProc2025}.

\subsection{Almost everywhere convergence of the Nörlund means of
Walsh-Fourier series}

Consider the general Nörlund means generated by 
\begin{equation*}
\mathbb{Q}:=\left\{ t_{k,n}:=\frac{q_{n-k}}{Q_{n}}:k=0,\dots ,n,\ n\in 
\mathbb{N}\right\} ,
\end{equation*}%
where $\{q_{n}:n\in \mathbb{N}\}$ is a non-increasing sequence of
nonnegative numbers and 
\begin{equation*}
Q_{n}:=\sum_{k=0}^{n}q_{k}.
\end{equation*}

\begin{theorem}
\label{T-Norl}Assume that $\{q_n\}_{n\ge0}$ is non-increasing,
$q_n\ge0$, and $Q_n:=\sum_{k=0}^n q_k>0$ for every $n$. The following two
statements hold.
\begin{enumerate}
\item[\textup{(a1)}] If
\begin{equation}
\sup_{n\in \mathbb{N}}\left( \frac{1}{Q_{2^{|n|}}}\sum_{s=0}^{|n|}Q_{2^{s}}%
\right) <\infty,  \label{Q-cond}
\end{equation}
then, for every $f\in L_{1}(\mathbb{I})$,
\begin{equation*}
\mathcal{T}_{n}^{\mathbb{Q}}(f) \longrightarrow f
\end{equation*}
almost everywhere on $\mathbb{I}$.

\item[\textup{(b1)}] If \eqref{Q-cond} fails, that is,
\begin{equation}
\sup_{n\in \mathbb{N}}\left( \frac{1}{Q_{2^{|n|}}}\sum_{s=0}^{|n|}Q_{2^{s}}%
\right) =\infty,  \label{div2}
\end{equation}
then there exists an integrable function $f\in L_{1}(\mathbb{I})$ such that
the Nörlund means $\mathcal{T}_{n}^{\mathbb{Q}}(f;x)$ diverge at every point
$x\in \mathbb{I}$. Thus everywhere divergence follows from the sole
condition \eqref{div2}, without any additional regularity hypothesis.
\end{enumerate}
\end{theorem}

\section{Proof of Theorem \protect\ref{conj1}}

\begin{proof}[Proof of Theorem~\protect\ref{conj1}(a)]
We begin with the representation  
\begin{equation*}
M_{n_a}(\boldsymbol{\Omega})f  = \sum_{k=0}^{\infty} 
\varepsilon_k(n_a)\,\Omega_k(n_a)\,\Delta_k(f w_{n_a})  =
\sum_{k=0}^{\infty}  \varepsilon_k(n_a)\,\Omega_k(n_a)  \bigl(E_{k+1}(f
w_{n_a}) - E_k(f w_{n_a})\bigr),  
\end{equation*}
where $E_k$ denotes conditional expectation with respect to the $k$-th
dyadic  $\sigma$-algebra and we use the convention $\varepsilon_{-1}(n_a)=0$%
,  $\Omega_{-1}(n_a)=0$.

Applying summation by parts, we obtain  
\begin{align*}
M_{n_a}(\boldsymbol{\Omega})f &= -\sum_{k=0}^{\infty} \bigl(%
\varepsilon_k(n_a)\,\Omega_k(n_a) - \varepsilon_{k-1}(n_a)\,\Omega_{k-1}(n_a)%
\bigr) E_k(f w_{n_a}) \\
&= \sum_{k=1}^{\infty} \bigl(\varepsilon_{k-1}(n_a) - \varepsilon_k(n_a)%
\bigr) \Omega_{k-1}(n_a)\,E_k(f w_{n_a}) \\
&\phantom{=} - \sum_{k=0}^{\infty} \varepsilon_k(n_a)\,\bigl(\Omega_k(n_a) -
\Omega_{k-1}(n_a)\bigr) E_k(f w_{n_a}).
\end{align*}

Define the dyadic martingale maximal operator 
\begin{equation*}
E^{\ast }(g):=\sup_{k\in \mathbb{N}}|E_{k}\left( g\right) |.
\end{equation*}%
It is well known (see, e.g., \cite{SWS}) that 
\begin{equation}
\Vert E^{\ast }\Vert _{L_{1}(\mathbb{I})\rightarrow L_{1,\infty }(\mathbb{I}%
)}<\infty .  \label{ww}
\end{equation}%
Since $|w_{n_{a}}|=1$, we have 
\begin{equation*}
|E_{k}(fw_{n_{a}})|\leq E^{\ast }(|f|)\quad \text{pointwise for all }k.
\end{equation*}

Using the monotonicity of the weights in $k$ and $|\varepsilon
_{k}(n_{a})|\leq 1$, we deduce 
\begin{align*}
|M_{n_{a}}(\boldsymbol{\Omega })f|& \leq E^{\ast }(|f|)\sum_{k=1}^{\infty }%
\bigl|\varepsilon _{k-1}(n_{a})-\varepsilon _{k}(n_{a})\bigr|\Omega
_{k-1}(n_{a}) \\
& \phantom{\le}\;+\;E^{\ast }(|f|)\sum_{k=0}^{\infty }\bigl(\Omega
_{k}(n_{a})-\Omega _{k-1}(n_{a})\bigr) \\
& =E^{\ast }(|f|)\sum_{k=1}^{|n_{a}|}\bigl|\varepsilon
_{k-1}(n_{a})-\varepsilon _{k}(n_{a})\bigr|\Omega _{k-1}(n_{a})\;+\;E^{\ast
}(|f|)\,\Omega _{|n_{a}|}(n_{a}).
\end{align*}

Taking the supremum over $a\in \mathbb{N}$, we obtain 
\begin{equation*}
W_{C}(\boldsymbol{\Omega },\{n_{a}\})f\leq E^{\ast }(|f|)\sup_{a\in \mathbb{N%
}}\sum_{k=1}^{|n_{a}|}\bigl|\varepsilon _{k-1}(n_{a})-\varepsilon _{k}(n_{a})%
\bigr|\Omega _{k-1}(n_{a})\;+\;E^{\ast }(|f|)\sup_{a\in \mathbb{N}}\Omega
_{|n_{a}|}(n_{a}).
\end{equation*}

By the uniform summability condition \eqref{cond} and the top-weight bound 
(see condition~\eqref{<1}, i.e.\ $\sup_{n}\Omega_{|n|}(n)<\infty$), both 
suprema on the right-hand side are finite. Hence  
\begin{equation*}
W_C(\boldsymbol{\Omega},\{n_a\})f \;\lesssim\; E^*(|f|)  \quad\text{pointwise%
},  
\end{equation*}
with an implied constant independent of $f$.

Combining this pointwise domination with the weak type $(1,1)$ estimate  %
\eqref{ww} for $E^*$ yields  
\begin{equation*}
\|W_C(\boldsymbol{\Omega},\{n_a\})f\|_{L_{1,\infty}(\mathbb{I})}  \lesssim
\|E^*(|f|)\|_{L_{1,\infty}(\mathbb{I})}  \lesssim \|f\|_{L_1(\mathbb{I})},  
\end{equation*}
which is exactly \eqref{weeek}. This completes the proof of part~\textup{(a)}%
.

Proof of Theorem~\ref{conj1}(b). Let
\begin{equation*}
\delta:=\sup_{n\in\mathbb P}\frac{\gamma(n)}{|n|}<1.
\end{equation*}
For each sufficiently large integer $a$, put
\begin{equation*}
G_a:=\inf\{\gamma(n):n\in\mathbb P,\ 2^{a/2}\le n<2^a\},
\qquad
\eta_a:=\min\left\{\left\lfloor\frac{a}{8}\right\rfloor,
\left\lfloor\frac{G_a}{2}\right\rfloor\right\}.
\end{equation*}
By Definition~\ref{dee}, $G_a\to\infty$, and hence $\eta_a\to\infty$.
Moreover, $2\eta_a<a/4$ for large $a$.

We shall use a block construction of length $\eta_a$. For a binary block
$\varepsilon_{a-2\eta_a},\dots,\varepsilon_{a-\eta_a-1}\in\{0,1\}$ set
\begin{equation*}
\lambda_a:=\sum_{j=0}^{\eta_a-1}\varepsilon_{a-2\eta_a+j}2^{a-2\eta_a+j},
\qquad
n_a:=\lambda_a+2^{\eta_a}\lambda_a.
\end{equation*}
Thus the same block of $\eta_a$ digits appears twice in the binary expansion
of $n_a$, once in the positions $a-2\eta_a,\dots,a-\eta_a-1$ and once in the
positions $a-\eta_a,\dots,a-1$.

Define the Walsh polynomial
\begin{equation}
W_a(t):=\frac{1}{\sqrt{\eta_a}}\Biggl(\prod_{j=0}^{\eta_a-1}
\bigl(1+r_{a-2\eta_a+j}(t)r_{a-\eta_a+j}(t)\bigr)\Biggr)
\Biggl(\sum_{l=a-2\eta_a}^{a-\eta_a-1}r_l(t)\Biggr). \label{W1-revised}
\end{equation}
Exactly as in the estimate above, the product in \eqref{W1-revised} is
supported on the dyadic set on which the two blocks of coordinates coincide,
and the Cauchy--Schwarz inequality gives
\begin{equation}
\|W_a\|_1\le 1. \label{W1-revised-L1}
\end{equation}

Let
\begin{equation*}
E_a:=\{x\in\mathbb I:x_{a-\eta_a}\dotplus x_{a-2\eta_a-1}=1\}.
\end{equation*}
Then $|E_a|=1/2$. For each fixed $x\in E_a$ choose the block digits as follows:
if
\begin{equation*}
\sum_{k=a-2\eta_a}^{a-\eta_a-1}x_k\ge \frac{\eta_a}{3},
\end{equation*}
put $\varepsilon_k=x_k$ on this block; otherwise put
$\varepsilon_k=1-x_k$. With this choice, at least a fixed positive proportion
of the terms in the block have the same sign. The corresponding integer will
be denoted by $n_a(x)$.

Because the chosen block contains at least one non-zero digit, for all large
$a$ we have
\begin{equation*}
2^{a-2\eta_a}\le n_a(x)<2^a.
\end{equation*}
Hence $\gamma(n_a(x))\ge G_a\ge 2\eta_a$. Since $|n_a(x)|\le a-1$, it follows
that
\begin{equation*}
|n_a(x)|-\gamma(n_a(x))\le a-1-2\eta_a<a-2\eta_a.
\end{equation*}
By the monotonicity of $k\mapsto\Omega_k(n)$ and by Definition~\ref{dee},
\begin{equation}
\Omega_{a-2\eta_a}(n_a(x))
\ge \Omega_{|n_a(x)|-\gamma(n_a(x))}(n_a(x))\ge c. \label{block-lower}
\end{equation}

The same convolution and orthogonality calculation as in the preceding block
construction gives, for $x\in E_a$,
\begin{equation}
M_{n_a(x)}(\boldsymbol{\Omega})W_a(x)
=\frac{w_{n_a(x)}(x)}{\sqrt{\eta_a}}
\sum_{k=a-2\eta_a}^{a-\eta_a-1}\varepsilon_k\Omega_k(n_a(x))r_k(x).
\label{block-orth}
\end{equation}
Indeed, the second copy of the digit block is killed by the defining
condition of $E_a$, while on the support of $W_a$ one has $w_{n_a(x)}=1$.
Using the choice of the digits, the monotonicity of the weights, and
\eqref{block-lower}, we obtain
\begin{equation}
|M_{n_a(x)}(\boldsymbol{\Omega})W_a(x)|
\gtrsim \frac{1}{\sqrt{\eta_a}}\sum_{a-2\eta_a\le k\le a-(5/3)\eta_a-1}
\Omega_k(n_a(x))
\gtrsim \sqrt{\eta_a},\qquad x\in E_a. \label{block-low}
\end{equation}
Since the maximal operator takes the supremum over all $n$, the dependence
of $n_a(x)$ on $x$ is harmless. Taking
$\lambda=c_0\sqrt{\eta_a}$ with $c_0>0$ sufficiently small, we get
\begin{equation*}
\|W_C(\boldsymbol{\Omega})W_a\|_{1,\infty}
\ge \lambda\,|\{x\in E_a:|M_{n_a(x)}(\boldsymbol{\Omega})W_a(x)|>\lambda\}|
\gtrsim \sqrt{\eta_a}.
\end{equation*}
Together with \eqref{W1-revised-L1} and $\eta_a\to\infty$, this gives
\begin{equation*}
\sup_a\frac{\|W_C(\boldsymbol{\Omega})W_a\|_{1,\infty}}{\|W_a\|_1}=\infty.
\end{equation*}
Thus $W_C(\boldsymbol{\Omega})$ is not of weak type~$(1,1)$. Theorem~\ref{conj1}
is proved.
\end{proof}

\section{Proof of Theorem \protect\ref{T2}}

\begin{proof}[Proof of Theorem~\ref{T2}]
We write $m:=|n|$. It is useful to introduce the width of the cone by
\begin{equation*}
h_\Delta(n):=
\begin{cases}
\omega_n, & \Delta=\Delta_\omega,\\[2mm]
\left\lfloor \dfrac{1-\kappa}{2}|n|\right\rfloor, & \Delta=\Delta_\kappa.
\end{cases}
\end{equation*}
Then $h_\Delta(n)\to\infty$ and $h_\Delta(n)/|n|$ is bounded away from one.
Moreover, if $1\le j\le h_\Delta(n)$, then all indices
$m-j+1,\dots,m$ belong to the cone $\Delta$ for all sufficiently large $|n|$.

Assume first that \eqref{sim} holds. For $1\le k\le m$ set
\begin{equation*}
\theta_k(n):=\frac{\Omega_k(n)-\Omega_{k-1}(n)}{\Omega_{k-1}(n)}
=\frac{\Omega_k(n)}{\Omega_{k-1}(n)}-1\ge0.
\end{equation*}
For large $n$ define $\gamma(n)$ to be the largest integer
$1\le j\le h_\Delta(n)$ such that
\begin{equation*}
\max_{m-j+1\le k\le m}\theta_k(n)\le \frac1j.
\end{equation*}
For the finitely many remaining values of $n$, define $\gamma(n):=1$. The
triangular limit \eqref{sim} implies that for every fixed $J$ and all
sufficiently large $|n|$,
\begin{equation*}
\max_{m-J+1\le k\le m}\theta_k(n)\le \frac1J,
\end{equation*}
and also $J\le h_\Delta(n)$. Hence $\gamma(n)\ge J$ for all large $|n|$.
Thus
\begin{equation*}
\lim_{N\to\infty}\inf\{\gamma(n): |n|\ge N\}=\infty.
\end{equation*}
Furthermore, $\gamma(n)\le h_\Delta(n)$, so
$\sup_n\gamma(n)/|n|<1$.

By telescoping,
\begin{equation*}
\frac{\Omega_m(n)}{\Omega_{m-\gamma(n)}(n)}
=\prod_{t=1}^{\gamma(n)}\bigl(1+\theta_{m-t+1}(n)\bigr).
\end{equation*}
The definition of $\gamma(n)$ and the inequality $\log(1+x)\le x$ give
\begin{equation*}
\log\frac{\Omega_m(n)}{\Omega_{m-\gamma(n)}(n)}
\le \sum_{t=1}^{\gamma(n)}\theta_{m-t+1}(n)\le 1.
\end{equation*}
Using the lower bound in the admissibility condition \eqref{<1}, we obtain
\begin{equation*}
\Omega_{|n|-\gamma(n)}(n)\ge e^{-1}\Omega_{|n|}(n)\ge c_0>0.
\end{equation*}
Therefore $\gamma$ is a sequence of divergence in the sense of
Definition~\ref{dee}.

Conversely, assume that a sequence of divergence exists. Since the ratio
limit in the definition of $\Delta$-admissibility exists and the weights are
nondecreasing in $k$, failure of \eqref{sim} means that
\begin{equation*}
\lim_{(k,n)_\Delta\to\infty}\frac{\Omega_k(n)}{\Omega_{k-1}(n)}=L>1.
\end{equation*}
Choose $\beta$ with $1<\beta<L$. For all sufficiently large $|n|$ and all
indices in the last $h_\Delta(n)$ levels of the cone, we have
\begin{equation*}
\frac{\Omega_k(n)}{\Omega_{k-1}(n)}\ge \beta.
\end{equation*}
Put
\begin{equation*}
j(n):=\min\{\gamma(n),h_\Delta(n)\}.
\end{equation*}
By Definition~\ref{dee} and by $h_\Delta(n)\to\infty$, we have
$j(n)\to\infty$ uniformly as $|n|\to\infty$. Hence
\begin{equation*}
\frac{\Omega_{|n|}(n)}{\Omega_{|n|-j(n)}(n)}
=\prod_{l=|n|-j(n)+1}^{|n|}\frac{\Omega_l(n)}{\Omega_{l-1}(n)}
\ge \beta^{j(n)}.
\end{equation*}
On the other hand, since $j(n)\le\gamma(n)$, monotonicity and
Definition~\ref{dee} imply
\begin{equation*}
\Omega_{|n|-j(n)}(n)\ge \Omega_{|n|-\gamma(n)}(n)\ge c,
\end{equation*}
while \eqref{<1} gives $\Omega_{|n|}(n)\le C$. Thus the last quotient is
bounded by $C/c$, contradicting $\beta^{j(n)}\to\infty$. Therefore the limit
must be equal to one.
\end{proof}

\section{\protect\bigskip Proof of Theorem \protect\ref{T3}}

\begin{proof}
We first prove that condition \eqref{kapa} implies \eqref{div22}. By Theorem~%
\ref{T2}, condition \eqref{kapa} yields the existence of a sequence of
divergence $\gamma = (\gamma(n))_{n\ge1}$ with respect to $\bigl(\Omega_k(n)%
\bigr)_{k,n\ge1}$. By definition of a sequence of divergence, we have 
\begin{equation*}
\sup_{n\in\mathbb{N}} \frac{\gamma(n)}{|n|} < 1 \quad\text{and}\quad
\Omega_{|n|-\gamma(n)}(n) \ge c > 0 \quad\text{for all } n\in\mathbb{N}.
\end{equation*}
Using the monotonicity in $k$ (i.e. $\Omega_{k-1}(n) \le \Omega_k(n)$), we
obtain for each $n$: 
\begin{equation*}
\sum_{k=0}^{|n|} \Omega_k(n) \;\ge\; \sum_{k=|n|-\gamma(n)}^{|n|}
\Omega_k(n) \;\ge\; \gamma(n)\,\Omega_{|n|-\gamma(n)}(n) \;\gtrsim\;
\gamma(n),
\end{equation*}
and since $\gamma(n)$ tends to infinity uniformly as $|n|\to\infty$, it follows that 
\begin{equation*}
\sup_{n\in\mathbb{N}} \sum_{k=0}^{|n|} \Omega_k(n) = \infty,
\end{equation*}
that is, \eqref{div22} holds.

\medskip

Next, we prove the converse: \eqref{div22} implies \eqref{kapa}. Assume that %
\eqref{div22} holds, but \eqref{kapa} fails. Then, along the cone $\Delta
_{\kappa }$, we have 
\begin{equation*}
\lim_{(k,n)_{\Delta _{\kappa }}\rightarrow \infty }\frac{\Omega _{k}(n)}{%
\Omega _{k-1}(n)}=L>1.
\end{equation*}%
Hence there exists a number $\beta $ with $1<\beta <L$ and an integer $%
N(\beta )$ such that 
\begin{equation}
\frac{\Omega _{k}(n)}{\Omega _{k-1}(n)}\geq \beta ,\qquad |n|>N(\beta ),\
(k,n)\in \Delta _{\kappa }.  \label{eq:beta-growth}
\end{equation}%
Equivalently, 
\begin{equation*}
\Omega _{k}(n)\geq \beta \,\Omega _{k-1}(n)\quad \text{for all }|n|>N(\beta
),\ \kappa |n|\leq k\leq |n|.
\end{equation*}

Fix $n>N(\beta )$ and any $m$ with $\kappa |n|\leq m\leq |n|$. Iterating %
\eqref{eq:beta-growth} from $k=m+1$ up to $k=|n|$ yields 
\begin{equation*}
\Omega _{|n|}(n)\;\geq \;\beta ^{|n|-m}\,\Omega _{m}(n),\quad \text{so that}%
\quad \Omega _{m}(n)\;\leq \;\beta ^{-(|n|-m)}\,\Omega _{|n|}(n).
\end{equation*}%
By the uniform bound on the top weights i.e. $\Omega _{|n|}(n)\lesssim 1$,
we obtain 
\begin{equation*}
\Omega _{m}(n)\;\lesssim \;\beta ^{-(|n|-m)},\qquad \kappa |n|\leq m\leq |n|.
\end{equation*}%
Consequently, 
\begin{equation*}
\sum_{m=\kappa |n|}^{|n|}\Omega _{m}(n)\;\lesssim \;\sum_{m=\kappa
|n|}^{|n|}\beta ^{-(|n|-m)}\;\;\lesssim \;1,
\end{equation*}%
with an implicit constant independent of $n$.

On the other hand, by monotonicity in $k$ we have, for $k<\kappa|n|$, 
\begin{equation*}
\Omega_k(n) \le \Omega_{\kappa|n|}(n),
\end{equation*}
so that 
\begin{equation*}
\sum_{k=0}^{\kappa|n|-1} \Omega_k(n) \;\le\;
\kappa|n|\,\Omega_{\kappa|n|}(n) \;\le\; \frac{\kappa}{1-\kappa}
\sum_{m=\kappa|n|}^{|n|} \Omega_m(n) \;\lesssim\; 1.
\end{equation*}
Therefore, 
\begin{equation*}
\sum_{k=0}^{|n|} \Omega_k(n) = \sum_{k=0}^{\kappa|n|-1} \Omega_k(n)
+\sum_{m=\kappa|n|}^{|n|} \Omega_m(n) \;\lesssim\; 1
\end{equation*}
for all sufficiently large $n$, which contradicts \eqref{div22}. This proves
the equivalence between \eqref{kapa} and \eqref{div22}.

\medskip

Finally, assume that $\{\Omega_k(n)\}$ is $\Delta_\omega$-admissible and that \eqref{omega} holds. The implication \eqref{omega} $\Rightarrow$ \eqref{div22} follows from Theorem~\ref{T2} and the preceding argument. To see that \eqref{omega} is not necessary, we construct a
family of weights $\{\Omega_k(n)\}_{k,n\ge1}$ such that 
\begin{align}
&\Omega_{k-1}(n) \le \Omega_k(n) \quad\text{for all } n\in\mathbb{N},\ k\ge1,
\label{zrd} \\
&\Omega_{|n|}(n) \sim 1 \quad\text{for all } n,  \label{eq1} \\
&\lim_{(k,n)_{\Delta_\omega}\to\infty} \frac{\Omega_k(n)}{\Omega_{k-1}(n)} =
L > 1,  \label{l}
\end{align}
and yet \eqref{div22} holds.

Fix a constant $L>1$ and define 
\begin{equation*}
\omega_n := \biggl\lfloor \frac{1}{2}\log_L(|n|+1)\biggr\rfloor.
\end{equation*}
For each $n$, define the weights $\Omega_k(n)$ for $0\le k\le |n|$ by 
\begin{equation*}
\Omega_k(n) := 
\begin{cases}
L^{-\omega_n}, & 0 \le k \le |n|-\omega_n, \\[4pt] 
L^{k-|n|}, & |n|-\omega_n < k \le |n|.%
\end{cases}%
\end{equation*}
Then $\Omega_k(n)$ is constant on the "early part" and increases
geometrically on the "tail part". It is immediate from the definition that %
\eqref{zrd} holds, and that 
\begin{equation*}
\Omega_{|n|}(n) = L^{|n|-|n|} = 1,
\end{equation*}
so \eqref{eq1} is satisfied. Moreover, for $|n|-\omega_n < k \le |n|$ we
have 
\begin{equation*}
\frac{\Omega_k(n)}{\Omega_{k-1}(n)} = \frac{L^{k-|n|}}{L^{k-1-|n|}} = L,
\end{equation*}
and the cone $\Delta_\omega$ lies entirely inside the "tail part" for $|n|$
large enough (since $|n|-\omega_n < k \le |n|$ exactly describes the cone
determined by $\omega$). Hence 
\begin{equation*}
\lim_{(k,n)_{\Delta_\omega}\to\infty} \frac{\Omega_k(n)}{\Omega_{k-1}(n)} =
L > 1,
\end{equation*}
so \eqref{omega} fails, but \eqref{l} holds.

We now show that \eqref{div22} nevertheless holds for this system. Set 
\begin{equation*}
\rho (n):=\sum_{k=0}^{|n|}\Omega _{k}(n)=\sum_{k=0}^{|n|-\omega _{n}}\Omega
_{k}(n)+\sum_{k=|n|-\omega _{n}+1}^{|n|}\Omega _{k}(n)=:\rho _{1}(n)+\rho
_{2}(n).
\end{equation*}%
For the tail part, 
\begin{equation*}
\rho _{2}(n)=\sum_{k=|n|-\omega _{n}+1}^{|n|}L^{k-|n|}=\sum_{j=0}^{\omega
_{n}-1}L^{-j}\leq \sum_{j=0}^{\infty }L^{-j}\lesssim 1.
\end{equation*}%
For the early part, 
\begin{equation*}
\rho _{1}(n)=\sum_{k=0}^{|n|-\omega _{n}}L^{-\omega _{n}}=L^{-\omega
_{n}}(|n|-\omega _{n}+1).
\end{equation*}%
By the choice of $\omega _{n}$, 
\begin{equation*}
\omega _{n}=\Bigl\lfloor\tfrac{1}{2}\log _{L}(|n|+1)\Bigr\rfloor\leq \tfrac{1%
}{2}\log _{L}(|n|+1),
\end{equation*}%
so 
\begin{equation*}
L^{-\omega _{n}}\geq L^{-\frac{1}{2}\log _{L}(|n|+1)}=(|n|+1)^{-1/2}.
\end{equation*}%
Thus 
\begin{equation*}
\rho _{1}(n)=L^{-\omega _{n}}(|n|-\omega _{n}+1)\gtrsim \frac{|n|-\omega
_{n}+1}{(|n|+1)^{1/2}}\gtrsim (|n|+1)^{1/2}\rightarrow \infty \quad \text{as 
}n\rightarrow \infty .
\end{equation*}%
Hence $\rho (n)=\rho _{1}(n)+\rho _{2}(n)\rightarrow \infty $, and therefore 
\begin{equation*}
\sup_{n\in \mathbb{N}}\sum_{k=0}^{|n|}\Omega _{k}(n)=\sup_{n\in \mathbb{N}%
}\rho (n)=\infty ,
\end{equation*}%
which is exactly \eqref{div22}.

This shows that \eqref{omega} is not necessary for \eqref{div22}, and completes the proof of Theorem~\ref{T3}.
\end{proof}

\section{Proof of Theorem \protect\ref{divergence}}

In this section we apply Theorem~\ref{conj1} to study the weak-type~$(1,1)$
boundedness of maximal operators associated with sequences of operators
generated by the partial sums of the Walsh--Fourier series. These sequences
arise naturally from matrix transformations of the Walsh--Fourier partial
sums and include, as particular cases, many classical summability methods
such as partial sums, Fejér means, Ces\`aro means, Riesz means, logarithmic
means and Nörlund means, among others. Throughout this section we restrict
our attention to the case in which the entries of the transformation matrix
form a non-decreasing sequence along each fixed row.

Below, we assume that $\left\{ a_{k}\right\} $ and $\left\{ b_{k}\right\} $
are two sequences satisfying the following properties:%
\begin{equation}
b_{k}+1<a_{k}-2\gamma \left( a_{k}\right) <a_{k}<\frac{b_{k+1}}{2}-1,
\label{s1}
\end{equation}%
\begin{equation}
a_{k}>3a_{k-1},  \label{s2}
\end{equation}%
\begin{equation}
\gamma \left( a_{k}\right) >k^{8}2^{4b_{k}},  \label{s3}
\end{equation}

\begin{equation}
\gamma \left( a_{k}\right) >\gamma \left( a_{k-1}\right) .  \label{s4}
\end{equation}

We note that the choice of $a_{k}$ does not depend on $b_{k}$ and $a_{k}$
can be chosen arbitrarily large. This argument will be used below in the
proofs of the lemmas, where the sequence $\left\{ a_{k}\right\} $ will be
subject to additional restrictions.

\begin{lemma}
\label{l1}Assume that the ratio condition \eqref{eqq2} holds and that $\left\{
a_{k}\right\} $ and the sequence $\left\{ \gamma (a_{k})\right\} $ satisfy
properties (\ref{s1})-(\ref{s3}). Then for each $k$ there is a Walsh
polynomial $W_{a_{k}}^{1}$, a measurable set $E^{1}\subset \mathbb{I}$ with $%
|E^{1}|=1$, and a function 
\begin{equation*}
n_{a_{k}}(\cdot ):E^{1}\rightarrow \mathbb{N}
\end{equation*}%
such that

\begin{enumerate}
\item $\func{sp}(W_{a_{k}}^{1})\subset \lbrack 2^{a_{k}-2\gamma \left(
a_{k}\right) },2^{a_{k}})\cap \mathbb{N}$;

\item $2^{a_{k}-2\gamma \left( a_{k}\right) }\leq n_{a_{k}}<2^{a_{k}}$ for
all $x\in E^{1}$;

\item $\displaystyle\sum_{k=1}^{\infty }\Vert W_{a_{k}}^{1}\Vert _{1}<\infty 
$;

\item for every $x\in E^{1}$ there exist infinitely many $k$ such that 
\begin{equation*}
\bigl|W_{a_{k}}^{1}\ast V_{n_{a_{k}}}^{\mathbb{T}}(x)\bigr|\gtrsim \sqrt[4]{%
\gamma (a_{k})}.
\end{equation*}
\end{enumerate}
\end{lemma}

\begin{proof}[Proof of Lemma \protect\ref{l1}]
By \eqref{eqq2} and the normalization of the matrix rows, the weights
$\Omega_s(n):=\widetilde T_{2^s,n}$ satisfy the admissibility hypotheses on
the cone $\Delta_\omega$ used in Theorem~\ref{T2}. Hence Theorem~\ref{T2}
gives a divergence scale for these weights. Applying the block construction
from the proof of Theorem~\ref{conj1}(b) to the weights
$\Omega_s(n)=\widetilde T_{2^s,n}$, we obtain the following objects:

\begin{enumerate}
	\item an increasing sequence $\{\gamma(a)\}_{a\in\mathbb{N}}\subset\mathbb{N}$
	with
	\[
	\sup_{a\in\mathbb{N}} \gamma(a) = \infty;
	\]
	\item a sequence of Walsh polynomials $W_{a}$ with
	\[
	\deg W_{a} < 2^{a}, \qquad a\in\mathbb{N};
	\]
	\item for every $a\in\mathbb{N}$ and every $x$ in the set
	\[
	E_{a}
	:= \bigl\{ x\in\mathbb{I} :
	x_{a-\gamma(a)} \dotplus x_{a-2\gamma(a)-1} = 1
	\bigr\},
	\]
	where $\dotplus$ denotes addition modulo $2$, there exists an integer
	$n_{a}(x)$ with
	\[
	2^{a-2\gamma(a)} \le n_{a}(x) < 2^{a}
	\]
	such that
	\begin{equation}\label{eq:A-large}
		\bigl|W_{a} * w_{n_{a}(x)}\,V_{n_{a}(x),3}^{\mathbb{T}}(x)\bigr|
		\ge c_{0}\,\sqrt{\gamma(a)}.
	\end{equation}
\end{enumerate}

Moreover, as is proved in \cite{GogNagyMathematics}, the weak type $(1,1)$
estimates for the operators $V_{n,1}^{\mathbb{T}}$ and $V_{n,2}^{\mathbb{T}}$
imply that, for every $\lambda >0$, 
\begin{equation}
\Bigl|\Bigl\{x\in \mathbb{I}:\sup_{n}\bigl|W_{a}\ast (w_{n}V_{n,1}^{\mathbb{T%
}}+w_{n}V_{n,2}^{\mathbb{T}})(x)\bigr|>\lambda \Bigr\}\Bigr|\lesssim \frac{\Vert
W_{a}\Vert _{1}}{\lambda }.  \label{eq:weak11}
\end{equation}

\medskip \noindent \textbf{Step 1: A full--measure set of points with
infinitely many \textquotedblleft good\textquotedblright\ indices.}

We assume that the sequence $\left\{ a_{k}\right\} $ satisfies conditions (%
\ref{s1})-(\ref{s4}). Next, we show that for almost every $x\in \mathbb{I}$
there exist infinitely many integers $k$ such that 
\begin{equation}
\bigl|W_{a_{k}}\ast V_{n_{a_{k}}}^{\mathbb{T}}(x)\bigr|\geq \frac{c_{0}}{2}%
\sqrt{\gamma (a_{k})}.  \label{eq:inf}
\end{equation}

Set 
\begin{equation*}
E:=\bigcap_{n=1}^{\infty }\bigcup_{k\geq n}E_{a_{k}}=\limsup_{k\rightarrow
\infty }E_{a_{k}}.
\end{equation*}%
By construction $|E_{a_{k}}|=\frac{1}{2}$ for all $k$. Since the defining
coordinates of the sets $E_{a_{k}}$ are disjoint, the sets $(E_{a_{k}})$ are
independent; hence, by the Borel--Cantelli lemma, 
\begin{equation*}
|E|=1,
\end{equation*}%
that is, for almost every $x\in \mathbb{I}$ we have $x\in E_{a_{k}}$ for
infinitely many $k$ (the choice of $k$ depends on $x$) and hence inequality (%
\ref{eq:A-large}) holds for infinitely many $k$.

Now, we define: 
\begin{equation*}
B_{a_{k}}(x):=\bigl|W_{a_{k}}\ast (w_{n_{a_{k}}}V_{n_{a_{k}},1}^{\mathbb{T}%
}+w_{n_{a_{k}}}V_{n_{a_{k}},2}^{\mathbb{T}})(x)\bigr|,\qquad A_{a_{k}}(x):=%
\bigl|W_{a_{k}}\ast w_{n_{a_{k}}}V_{n_{a_{k}},3}^{\mathbb{T}}(x)\bigr|,
\end{equation*}%
and define 
\begin{equation*}
B_{a_{k}}:=\Bigl\{x\in \mathbb{I}:B_{a_{k}}(x)\geq \frac{c_{0}}{2}\sqrt{%
\gamma (a_{k})}\Bigr\}.
\end{equation*}%
Clearly, 
\begin{equation*}
B_{a_{k}}\subset \Bigl\{x\in \mathbb{I}:\sup_{k}B_{a_{k}}(x)>\frac{c_{0}}{2}%
\sqrt{\gamma (a_{k})}\Bigr\}.
\end{equation*}%
Applying the weak type $(1,1)$ estimate \eqref{eq:weak11} with $\lambda =%
\frac{c_{0}}{2}\sqrt{\gamma (a_{k})}$ gives 
\begin{equation*}
|B_{a_{k}}|\lesssim \frac{\Vert W_{a_{k}}\Vert _{1}}{\sqrt{\gamma (a_{k})}}.
\end{equation*}%
By the choice of $\gamma (a_{k})$ in (\ref{s3}) and the bound $\|W_{a_k}\|_1\le1$, we
may assume that $\sum_{k}|B_{a_{k}}|<\infty $. Another application of the
Borel--Cantelli lemma yields 
\begin{equation*}
B:=\limsup_{k\rightarrow \infty }B_{a_{k}}=\bigcap_{n=1}^{\infty
}\bigcup_{k\geq n}B_{a_{k}}
\end{equation*}%
has measure $|B|=0$. Equivalently, for almost every $x\in \mathbb{I}$ there
exists $N(x)$ such that 
\begin{equation}
B_{a_{k}}(x)<\frac{c_{0}}{2}\sqrt{\gamma (a_{k})}\qquad \text{for all }k\geq
N(x).  \label{eq:inv}
\end{equation}

Now define the full-measure set 
\begin{equation*}
E^{1}:=E\setminus B.
\end{equation*}%
Then $|E^{1}|=1$, and for each $x\in E^{1}$ we have:

\begin{itemize}
\item $x\in E_{a_{k}}$ for infinitely many integers $k$;

\item there exists $N(x)$ such that \eqref{eq:inv} holds for all $k\geq N(x)$%
.
\end{itemize}

Fix such an $x$ and choose $k\geq N(x)$ with $x\in E_{a_{k}}$; there are
infinitely many such $k$. For these indices we combine \eqref{eq:A-large}
and \eqref{eq:inv}. From (\ref{V=V1+V2+V3}) \ we have 
\begin{equation*}
\bigl|W_{a_{k}}\ast V_{n_{a_{k}}}^{\mathbb{T}}(x)\bigr|%
=|A_{a_{k}}(x)+B_{a_{k}}(x)|\geq |A_{a_{k}}(x)|-|B_{a_{k}}(x)|.
\end{equation*}%
By \eqref{eq:A-large} and \eqref{eq:inv} we obtain, for all sufficiently
large $k$ with $x\in E_{a_{k}}$, 
\begin{equation*}
|A_{a_{k}}(x)|\geq c_{0}\sqrt{\gamma (a_{k})}\quad \text{and}\quad
|B_{a_{k}}(x)|<\frac{c_{0}}{2}\sqrt{\gamma (a_{k})}.
\end{equation*}%
Thus 
\begin{equation}
\bigl|W_{a_{k}}\ast V_{n_{a_{k}}}^{\mathbb{T}}(x)\bigr|\geq c_{0}\sqrt{\gamma (a_k)}-\frac{c_{0}}{2}\sqrt{\gamma (a_{k})}=\frac{c_{0}}{2}\sqrt{\gamma
(a_{k})}  \label{inf}
\end{equation}%
for some absolute constant $c_{0}>0$, and for infinitely many $k$. This
proves \eqref{eq:inf} for almost every $x\in \mathbb{I}$.

\medskip \noindent \textbf{Step 2: Construction of the polynomials $%
W_{a_{k}}^{1}$.}

Let $(a_{k})$ be any strictly increasing sequence of indices with conditions
(\ref{s1})-(\ref{s4}) such that for each $x\in E^{1}$ the inequality %
\eqref{eq:inf} holds for infinitely many $k$ (such a sequence exists by the
previous step). For each $k\in \mathbb{N}$ define 
\begin{equation*}
W_{a_{k}}^{1}(x):=\frac{W_{a_{k}}\left( x\right) }{\sqrt[4]{\gamma (a_{k})}},
\end{equation*}%
Then it is easy to see that%
\begin{equation*}
\text{sp}\left( W_{a_{k}}^{1}\right) \subset \left[ 2^{a_{k}-2\gamma \left(
a_{k}\right) },2^{a_{k}}\right) \cap \mathbb{N}
\end{equation*}%
and%
\begin{equation*}
\sum\limits_{k}\left\Vert W_{a_{k}}^{1}\right\Vert \leq \sum\limits_{k}\frac{%
1}{\sqrt[4]{\gamma \left( a_{k}\right) }}\leq \sum\limits_{k}\frac{1}{k^{2}}%
<\infty .
\end{equation*}%
Using inequality (\ref{inf}) we get%
\begin{equation*}
\left\vert \mathcal{T}_{n_{a_{k}}}^{\mathbb{T}}\left( W_{a_{k}}^{1},x\right)
\right\vert =\frac{\left\vert \mathcal{T}_{n_{a_{k}}}^{\mathbb{T}}\left(
W_{a_{k}},x\right) \right\vert }{\sqrt[4]{\gamma \left( a_{k}\right) }}%
\gtrsim \sqrt[4]{\gamma \left( a_{k}\right) }\text{ \ \ (for infinitely many
integers }k,\text{ depending on }x\text{).}
\end{equation*}

This completes the proof of Lemma \ref{l1}.
\end{proof}

\begin{lemma}
\label{l2} Let $E^{0}\subset \mathbb{I}$ be a measurable set with $|E^{0}| =
0$, and let $\{\varepsilon_{j} : j\in\mathbb{N}\}$ be a sequence of positive
numbers such that $\varepsilon_{j}\to 0$ as $j\to\infty$. Then there exists
an increasing sequence $\{b_{j} : j\in\mathbb{N}\}$ with $%
b_{j}\uparrow\infty $ as $j\to\infty$. Moreover, the choice of $b_{j}$
depends on $\varepsilon_{j}$ and can be made arbitrarily large by taking $%
\varepsilon_{j} $ sufficiently small.

In addition, there exists a sequence $\{A_{j}:j\in \mathbb{N}\}$ of
measurable subsets of $\mathbb{I}$ and a sequence of Walsh polynomials $%
\{W_{b_{j}}^{0}:j\in \mathbb{N}\}$ such that:

\begin{enumerate}
\item $E^{0}\subset \bigcup_{j=1}^{\infty} A_{j}$;

\item each $x\in E^{0}$ belongs to infinitely many of the sets $A_{j}$;

\item $|A_{j}| < \varepsilon_{j}$ for all $j\in\mathbb{N}$;

\item $\func{sp}(W_{b_{j}}^{0})\subset \lbrack 2^{b_{j}},2^{b_{j}+1})$ for
all $j\in \mathbb{N}$;

\item for all $x\in A_{j},\ j\in \mathbb{N};$ 
\begin{equation*}
\bigl|\mathcal{T}_{2^{b_{j}+2}}^{\mathbb{T}}(W_{b_{j}}^{0};x)\bigr|\geq
2^{b_{j}/2};\qquad 
\end{equation*}

\item 
\begin{equation*}
\sum_{j=0}^{\infty }\Vert W_{b_{j}}^{0}\Vert _{1}<\infty .
\end{equation*}
\end{enumerate}
\end{lemma}

\begin{proof}[Proof of Lemma \protect\ref{l2}]
The existence of a family $\{A_{j}\}$ satisfying properties (1)-(3) is a
standard consequence of the fact that $|E^{0}|=0$ (see, for example, \cite%
{SWS}). More precisely, there exists a sequence of dyadic intervals $%
\{I_{k}\}_{k\ge 1}$ such that 
\begin{equation*}
E^{0} \subset \bigcup_{k=1}^{\infty} I_{k}, \qquad \sum_{k=1}^{\infty}
|I_{k}| < 1,
\end{equation*}
and each $x\in E^{0}$ belongs to infinitely many of the intervals $I_{k}$.
We may represent each $A_{j}$ as a finite union of these intervals: 
\begin{equation*}
A_{j} := \bigcup_{k=n_{j}}^{n_{j+1}-1} I_{k},
\end{equation*}
for a suitably chosen increasing sequence of indices $\{n_{j}\}$.

Since $\sum_{k=1}^{\infty }|I_{k}|<1$, we can choose $n_{1}<n_{2}<\cdots $
such that 
\begin{equation*}
|A_{j}|=\sum_{k=n_{j}}^{n_{j+1}-1}|I_{k}|\leq \sum_{k=n_{j}}^{\infty
}|I_{k}|<\varepsilon _{j},\qquad j\in \mathbb{N}.
\end{equation*}%
This yields property~(3). Property~(1) follows from $E^{0}\subset
\bigcup_{k}I_{k}$ and the fact that $\{n_{j}\}$ partitions the index set $%
\mathbb{N}$ into disjoint blocks, while property~(2) follows from the
assumption that each $x\in E^{0}$ belongs to infinitely many of the
intervals $I_{k}$.

Next we define the sequence $\{b_{j}\}$. Set 
\begin{equation}
2^{b_{j}} := \max\Bigl\{\frac{1}{|I_{k}|} : n_{j}\le k<n_{j+1}\Bigr\}.
\label{bj}
\end{equation}
Since $|I_{k}|\to 0$ as $k\to\infty$ (because $\sum_{k} |I_{k}|<\infty$),
equality~\eqref{bj} implies that $b_{j}\to\infty$ as $j\to\infty$. Moreover,
by refining the choice of the indices $\{n_{j}\}$ (equivalently, by taking $%
\varepsilon_{j}$ smaller), we can make $b_{j}$ arbitrarily large, which
establishes the claimed dependence of $b_{j}$ on $\varepsilon_{j}$.

\medskip

We now turn to the construction of the polynomials $W_{j}^{0}$. Define 
\begin{equation*}
\widetilde{W}_{b_{j}}^{0}(x):=\alpha _{j}\,r_{b_{j}}(x)\,\mathds{1}%
_{A_{j}}(x)=\sum_{k=2^{b_{j}}}^{2^{b_{j}+1}-1}C_{k}^{(j)}w_{k}(x),
\end{equation*}%
where $\alpha _{j}>0$ will be chosen later, $r_{b_{j}}$ denotes the $b_{j}$%
-th Rademacher function, and $w_{k}$ is the $k$-th Walsh function. Thus $%
\widetilde{W}_{j}^{0}$ is a Walsh polynomial, and its Walsh spectrum is
contained in $[2^{b_{j}},2^{b_{j}+1})$.

We define the desired polynomial by 
\begin{equation*}
W_{b_{j}}^{0}(x):=\sum_{k=2^{b_{j}}}^{2^{b_{j}+1}-1}\frac{C_{k}^{(j)}}{%
T_{2^{b_{j}+2}}^{\left( k+1\right) }}\,w_{k}(x).
\end{equation*}

The denominator is uniformly bounded below. Indeed, since each row of
$\mathbb{T}$ is nondecreasing and has sum one, the mass of the last
$n-k$ terms is at least $(n-k)/(n+1)$. With $n=2^{b_j+2}$ and
$2^{b_j}\le k<2^{b_j+1}$, this gives
\begin{equation*}
T_{2^{b_j+2}}^{(k+1)}=\sum_{l=k+1}^{2^{b_j+2}}t_{l,2^{b_j+2}}
\ge \frac{2^{b_j+2}-k}{2^{b_j+2}+1}\ge c>0,
\end{equation*}
where $c$ is an absolute constant independent of $j$ and $k$. Hence 
\begin{equation*}
\Bigl\|W_{b_{j}}^{0}\Bigr\|_{2}^{2}=\sum_{k=2^{b_{j}}}^{2^{b_{j}+1}-1}\Bigl(%
\frac{C_{k}^{(j)}}{T_{2^{b_{j}+2}}^{\left( k+1\right) }}\Bigr)^{2}\lesssim
\sum_{k=2^{b_{j}}}^{2^{b_{j}+1}-1}\bigl(C_{k}^{(j)}\bigr)^{2}=\Bigl\|%
\widetilde{W}_{b_{j}}^{0}\Bigr\|_{2}^{2}.
\end{equation*}%
By Parseval's identity and the definition of $\widetilde{W}_{b_{j}}^{0}$, 
\begin{equation*}
\Bigl\|\widetilde{W}_{b_{j}}^{0}\Bigr\|_{2}^{2}=\alpha _{j}^{2}|A_{j}|,
\end{equation*}%
so that 
\begin{equation*}
\Vert W_{b_{j}}^{0}\Vert _{1}\leq \Vert W_{b_{j}}^{0}\Vert _{2}\lesssim
\alpha _{j}\,|A_{j}|^{1/2}.
\end{equation*}

The sequence $\{a_{j}\}$ of positive integers was constructed according to
conditions (\ref{s1})-(\ref{s3}), and we define 
\begin{equation*}
\varepsilon _{j}:=2^{-2a_{j}},\qquad \alpha _{j}:=2^{a_{j}/2}
\end{equation*}
and hence 
\begin{equation*}
\alpha _{j}\,|A_{j}|^{1/2}\leq 2^{a_{j}/2}\cdot 2^{-a_{j}}=2^{-a_{j}/2}.
\end{equation*}%
Thus 
\begin{equation*}
\Vert W_{b_{j}}^{0}\Vert _{1}\leq \Vert W_{b_{j}}^{0}\Vert _{2}\lesssim
2^{-a_{j}/2}
\end{equation*}%
and 
\begin{equation*}
\sum_{j=0}^{\infty }\Vert W_{b_{j}}^{0}\Vert _{1}<\infty ,
\end{equation*}%
which proves property~(6).

It remains to verify property~(5). Using the definition of $\mathcal{T}_{n}^{%
\mathbb{T}}$ and the fact that $\mathcal{T}_{n}^{\mathbb{T}}$ acts
diagonally on Walsh polynomials, we have 
\begin{align*}
\mathcal{T}_{2^{b_{j}+2}}^{\mathbb{T}}(W_{b_{j}}^{0};x)&
=\sum_{k=0}^{2^{b_{j}+2}}t_{k,2^{b_{j}+2}}\,S_{k}(W_{b_{j}}^{0};x) \\
& =\sum_{l=2^{b_{j}}}^{2^{b_{j}+1}-1}\frac{C_{l}^{(j)}}{T_{2^{b_{j}+2}}^{%
\left( l+1\right) }}\sum_{k=0}^{2^{b_{j}+2}}t_{k,2^{b_{j}+2}}S_{k}(w_{l};x).
\end{align*}%
Since $S_{k}(w_{l};x)=0$ for $k\leq l$ and $S_{k}(w_{l};x)=w_{l}(x)$ for $k>l
$, we obtain 
\begin{equation*}
\sum_{k=0}^{2^{b_{j}+2}}t_{k,2^{b_{j}+2}}S_{k}(w_{l};x)=\Bigl(%
\sum_{k=l+1}^{2^{b_{j}+2}}t_{k,2^{b_{j}+2}}\Bigr)w_{l}(x)=T_{2^{b_{j}+2}}^{%
\left( l+1\right) }\,w_{l}(x).
\end{equation*}%
Hence 
\begin{align*}
\mathcal{T}_{2^{b_{j}+2}}^{\mathbb{T}}(W_{j}^{0};x)&
=\sum_{l=2^{b_{j}}}^{2^{b_{j}+1}-1}\frac{C_{l}^{(j)}}{T_{2^{b_{j}+2}}^{%
\left( l+1\right) }}T_{2^{b_{j}+2}}^{\left( l+1\right) }\,w_{l}(x) \\
& =\sum_{l=2^{b_{j}}}^{2^{b_{j}+1}-1}C_{l}^{(j)}w_{l}(x)=\widetilde{W}%
_{b_{j}}^{0}(x)=\alpha _{j}\,r_{b_{j}}(x)\,\mathds{1}_{A_{j}}(x) \\
& =2^{a_{j}/2}r_{b_{j}}(x)\,\mathds{1}_{A_{j}}(x).
\end{align*}%
In particular, for $x\in A_{j}$ we have $|r_{b_{j}}(x)|=1$, and therefore 
\begin{equation*}
\bigl|\mathcal{T}_{2^{b_{j}+2}}^{\mathbb{T}}(W_{b_{j}}^{0};x)\bigr|%
=2^{a_{j}/2}\geq 2^{b_{j}/2},\qquad x\in A_{j},
\end{equation*}%
which is property~(5).

This completes the proof of Lemma~\ref{l2}.
\end{proof}

\begin{proof}[Proof of Theorem \protect\ref{divergence}]
We now construct a function $f_{0}$ as follows%
\begin{equation*}
f_{0}\left( x\right) :=\sum\limits_{k=0}^{\infty }\left( W_{b_{k}}^{0}\left(
x\right) +W_{a_{k}}^{1}\left( x\right) \right) .
\end{equation*}%
Using Lemma \ref{l1} and Lemma \ref{l2} we conclude that $f_{0}\in
L_{1}\left( \mathbb{I}\right) $. Now, we assume that $x$ is arbitrary from $%
\mathbb{I}$. Then since $\mathbb{I}=E^{0}\cup E^{1}$ we have that $x$
belongs either to $E^{0}$ or to $E^{1}$. First, we assume that $x\in E^{1}$. Then
for infinitely many $k$ inequality (\ref{eq:inf}) holds. By Lemmas \ref{l1}
and \ref{l2}, it follows that 
\begin{eqnarray*}
\mathcal{T}_{n_{a_{k}}}^{\mathbb{T}}\left( f_{0},x\right) 
&=&\sum\limits_{j=0}^{n_{a_{k}}}t_{j,n_{a_{k}}}S_{j}\left( f_{0},x\right)  \\
&=&\sum\limits_{l=0}^{k}\sum\limits_{j=0}^{n_{a_{k}}}t_{j,n_{a_{k}}}S_{j}%
\left( W_{b_{l}}^{0}+W_{a_{l}}^{1},x\right)  \\
&=&\sum\limits_{l=0}^{k-1}\sum%
\limits_{j=2^{b_{l}}}^{n_{a_{k}}}t_{j,n_{a_{k}}}S_{j}\left(
W_{b_{l}}^{0}+W_{a_{l}}^{1},x\right)
+\sum\limits_{j=0}^{n_{a_{k}}}t_{j,n_{a_{k}}}S_{j}\left(
W_{b_{k}}^{0},x\right)
+\sum\limits_{j=0}^{n_{a_{k}}}t_{j,n_{a_{k}}}S_{j}\left(
W_{a_{k}}^{1},x\right) .
\end{eqnarray*}%
Consequently, by (\ref{s3})%
\begin{eqnarray*}
\left\vert \mathcal{T}_{n_{a_{k}}}^{\mathbb{T}}\left( f_{0},x\right)
\right\vert  &\gtrsim &\mathcal{T}_{n_{a_{k}}}^{\mathbb{T}}\left(
W_{a_{k}}^{1},x\right) -2^{b_{k}}-k2^{a_{k-1}} \\
&\gtrsim &\sqrt[4]{\gamma \left( a_{k}\right) }-2^{b_{k}}-k2^{a_{k-1}}%
\gtrsim \sqrt[4]{\gamma \left( a_{k}\right) }\text{,}
\end{eqnarray*}%
for infinitely many integers $k$. \newline
Suppose that $x\in E^{0}$. Then $x$ belongs to $A_{k}$ for
infinitely many integers $k$. Consequently, we can write 
\begin{eqnarray*}
\mathcal{T}_{2^{b_{k}+2}}^{\mathbb{T}}\left( f_{0},x\right) 
&=&\sum\limits_{j=0}^{2^{b_{k}+2}}t_{j,2^{b_{k}+2}}S_{j}\left(
f_{0},x\right)  \\
&=&\sum\limits_{j=0}^{2^{b_{k}+2}}t_{j,2^{b_{k}+2}}S_{j}\left(
W_{b_{k}}^{0},x\right)
+\sum\limits_{l=0}^{k-1}\sum%
\limits_{j=0}^{2^{b_{k}+2}}t_{j,2^{b_{k}+2}}S_{j}\left(
W_{a_{l}}^{1},x\right)  \\
&=&\mathcal{T}_{2^{b_{k}+2}}^{\mathbb{T}}\left( W_{b_{k}}^{0},x\right)
+\sum\limits_{l=0}^{k-1}\sum%
\limits_{j=2^{a_{l}}}^{2^{b_{k}+2}}t_{j,2^{b_{k}+2}}S_{2^{a_{l}}}\left(
W_{a_{l}}^{1},x\right) .
\end{eqnarray*}%
Consequently, from (\ref{s1}) and Lemma \ref{l2}, we have%
\begin{equation*}
\left\vert \mathcal{T}_{2^{b_{k}+2}}^{\mathbb{T}}\left( f_{0},x\right)
\right\vert \gtrsim 2^{b_{k}/2}-2^{a_{k-1}}\gtrsim 2^{a_{k-1}}.
\end{equation*}%
This completes the proof of Theorem \ref{divergence}.
\end{proof}

\section{Proof of Theorem \protect\ref{VP-W}}

Define the matrix $\mathbb{V} = \{t_{k,n}\}$ by 
\begin{equation*}
t_{k,n} := 
\begin{cases}
0, & k = 0,\dots,n-\lambda_n-1, \\[4pt] 
\dfrac{1}{\lambda_n+1}, & k = n-\lambda_n,\dots,n,%
\end{cases}
\qquad n\in\mathbb{N}. 
\end{equation*}
Then the corresponding matrix means $\mathcal{T}_n^{\mathbb{V}}$ coincide
with the de la Vallée Poussin means $V_n^{(\lambda)}$ in \eqref{VP}.

For this matrix one checks (using the definition of $\widetilde{T}_{2^s,n}$)
that 
\begin{equation}  \label{eq:VP-Ttilde}
\widetilde{T}_{2^{s},n} \lesssim 
\begin{cases}
\dfrac{2^{s}}{\lambda_{n}+1}, & s = 0,1,\dots,|\lambda_{n}|, \\[6pt] 
\dfrac{\lambda_{n}}{\lambda_{n}+1}, & s = |\lambda_{n}|+1,\dots,|n|.%
\end{cases}%
\end{equation}

\subsection*{Proof of part \textup{(a)}}

Assume that condition \eqref{leid} holds, i.e. 
\begin{equation*}
n = O(\lambda_n) \quad (n\to\infty). 
\end{equation*}
Summing the bounds in \eqref{eq:VP-Ttilde}, we obtain 
\begin{align*}
\sum_{s=0}^{|n|} \widetilde{T}_{2^{s},n} &\lesssim \frac{1}{\lambda_n+1}
\sum_{s=0}^{|\lambda_n|} 2^{s} \;+\; \sum_{s=|\lambda_n|+1}^{|n|} \frac{%
\lambda_n}{\lambda_n+1} \\
&\lesssim \frac{2^{|\lambda_n|+1}}{\lambda_n+1} \;+\; (|n|-|\lambda_n|)\,%
\frac{\lambda_n}{\lambda_n+1}.
\end{align*}
Since $2^{|\lambda_n|} \simeq \lambda_n$ and $\lambda_n/(\lambda_n+1)\le 1$,
this yields 
\begin{equation*}
\sum_{s=0}^{|n|} \widetilde{T}_{2^{s},n} \lesssim 1 + |n|-|\lambda_n|. 
\end{equation*}
Moreover, $|n|\simeq \log_2 n$ and $|\lambda_n|\simeq \log_2 \lambda_n$, so 
\begin{equation*}
|n| - |\lambda_n| \lesssim \log\frac{n}{\lambda_n}. 
\end{equation*}
By \eqref{leid}, there exists $C\ge1$ such that $n/\lambda_n \le C$ for all
sufficiently large $n$, and hence 
\begin{equation*}
\sup_{n\in\mathbb{N}} \sum_{s=0}^{|n|} \widetilde{T}_{2^{s},n} < \infty. 
\end{equation*}

Thus the uniform boundedness condition in Theorem~\ref{divergence}\,(1) is
satisfied for the matrix $\mathbb{V}$. Consequently, for every $f\in L_{1}(%
\mathbb{I})$ one has 
\begin{equation*}
\mathcal{T}_{n}^{\mathbb{V}}(f;x) = V_{n}^{(\lambda)}(f;x) \longrightarrow
f(x) \quad \text{as } n\to\infty, 
\end{equation*}
for almost every $x\in\mathbb{I}$. This proves part~\textup{(a)} of Theorem~%
\ref{VP-W}.

\subsection*{Proof of part \textup{(b)}}

Now assume that 
\begin{equation}
\frac{n}{\lambda _{n}}\longrightarrow \infty \quad (n\rightarrow \infty ),
\label{div55-again}
\end{equation}
Set
\begin{equation*}
\omega_n:=\left\lfloor\frac{|n|-|\lambda_n|}{2}\right\rfloor .
\end{equation*}
By \eqref{div55-again}, we have $\omega_n\to\infty$ and
$\omega_n/|n|\to0$. Moreover, if $(s,n)\in\Delta_\omega$, then, for all
large $n$, $s>|\lambda_n|$; hence the second line of \eqref{eq:VP-Ttilde}
applies
\begin{equation*}
\lim_{(s,n)_{\Delta _{\omega }}\rightarrow \infty }\frac{\widetilde{T}%
_{2^{s},n}}{\widetilde{T}_{2^{s-1},n}}=1.
\end{equation*}%
In other words, the triangular ratio condition \eqref{eqq2} of Theorem~\ref%
{divergence}\thinspace (2) is satisfied in this case.

Therefore, by Theorem~\ref{divergence}\,(2) there exists an integrable
function $f\in L_{1}(\mathbb{I})$ such that 
\begin{equation*}
\limsup_{n\to\infty} \bigl|\mathcal{T}_{n}^{\mathbb{V}}(f;x)\bigr|
= \infty 
\end{equation*}
at every point $x\in\mathbb{I}$. Since $\mathcal{T}_{n}^{\mathbb{V}%
}(f;x)=V_{n}^{(\lambda)}(f;x)$, this is precisely the statement of part~%
\textup{(b)} of Theorem~\ref{VP-W}.

The proof of Theorem~\ref{VP-W} is complete. \qedhere

\section{Proof of Theorem \protect\ref{T-Norl}}

We now observe that the triangular limit condition %
\eqref{limmE2} is not required in the special case when
the weights $\Omega _{k}(n)$ originate from a single nondecreasing
sequence $\{Q_{2^{k}}\}_{k\geq 0}$. Indeed, suppose that for every $n\in 
\mathbb{N}$ we define 
\begin{equation}
\Omega _{k}(n):=\frac{Q_{2^{k}}}{Q_{2^{|n|}}},\qquad 0\leq k\leq |n|,
\label{eq:OmegaQ}
\end{equation}%
where the sequence $\{Q_{2^{k}}\}_{k\geq 0}$ satisfies 
\begin{equation*}
0\leq Q_{2^{k}}\leq Q_{2^{k+1}}\qquad (k\geq 0).
\end{equation*}%
In order to prove Theorem \ref{T-Norl} by Theorems \ref{conj1} and %
\ref{divergence} it is sufficient to show that there exists a sequence of divergence $\left( \gamma _{j}\right) _{j=1}^{\infty }$.

\begin{prop}
\label{prop2} Let $a_{k}:=Q_{2^{k}}$ and $A_{k}:=\sum_{j=0}^{k} a_{j}$.
Assume that 
\begin{equation}  \label{div}
\sup_{n\ge 0}\frac{A_{n}}{a_{n}} = \infty .
\end{equation}
Then there are a constant $c\in(0,1)$, an increasing sequence $%
n_{j}\uparrow\infty$, and integers $\gamma_{j}\uparrow\infty$ such that 
\begin{equation*}
a_{\,n_{j}-\gamma_{j}} \;\ge\; c\,a_{\,n_{j}} \qquad (j\in\mathbb{N}).
\end{equation*}
\end{prop}

\begin{proof}[Proof of Proposition \protect\ref{prop2}]
We first assume that 
\begin{equation}  \label{goal}
\forall m\in\mathbb{N},\ \text{there exist infinitely many }n \ \text{such
that }\ a_{n-m}\ge \tfrac12\,a_{n}.
\end{equation}
Fix $m\in\mathbb{N}$ and choose $n_{m}$ so large that $n_{1}<n_{2}<\cdots$
and $a_{n_{m}-m}\ge \tfrac12\,a_{n_{m}}$. Setting $\gamma_{m}:=m$, we obtain 
$\gamma_{m}\uparrow\infty$ and the desired estimate along the subsequence $%
\{n_{m}\}$.

\smallskip Next we show that \eqref{goal} must hold. Assume to the contrary
that it fails. Then its negation is

\begin{equation}  \label{goal-neg}
\exists\,m\in\mathbb{N} \ \text{such that only finitely many }n\ \text{%
satisfy}\ a_{n-m}\ge \tfrac12\,a_{n}.
\end{equation}

Thus there exist integers $m$ and $N$ such that 
\begin{equation}  \label{goal2}
\forall n>N:\quad a_{\,n-m} < \tfrac12\, a_{\,n}.
\end{equation}
We now show that \eqref{goal2} contradicts \eqref{div}.

\smallskip Decomposing the sum $A_{n}$ into $m$ arithmetic progressions
modulo $m$, for any $n$ we may write 
\begin{equation}  \label{An-decomp}
A_{n} \;\le\; C + \sum_{l=0}^{m-1}\ \sum_{0\le j\le (n-l)/m} a_{\,n-l-jm},
\end{equation}
where $C:=\sum_{k=0}^{N} a_{k}$.

Iterating \eqref{goal2}, for all $j\geq 0$ with $n-l-jm>N$ we have 
\begin{equation*}
a_{\,n-l-jm}\;\leq \;\left( \tfrac{1}{2}\right) ^{j}a_{\,n-l}\;\leq \;\left( 
\tfrac{1}{2}\right) ^{j}a_{n}.
\end{equation*}%
Substituting into \eqref{An-decomp} yields 
\begin{equation*}
A_{n}\;\leq \;C+a_{n}\sum_{l=0}^{m-1}\sum_{j=0}^{\infty }\left( \tfrac{1}{2}%
\right) ^{j}\;=\;C+2m\,a_{n}.
\end{equation*}%
Consequently, 
\begin{equation*}
\sup_{n}\frac{A_{n}}{a_{n}}\;\leq \;\frac{C}{a_{n}}+2m\;\leq \;\frac{C}{a_{N}%
}+2m\;<\infty ,
\end{equation*}%
contradicting \eqref{div}. Thus \eqref{goal-neg} and hence \eqref{goal2} are
impossible. Therefore \eqref{goal} holds, and the proposition follows.

\end{proof}

\begin{proof}[Completion of the proof of Theorem~\ref{T-Norl}]
Part~\textup{(a1)} is precisely the uniform boundedness criterion for
Nörlund means in the Walsh setting proved in~\cite{GogMukh2024}. Indeed, for
$\mathbb Q$ one has, up to absolute constants,
\begin{equation*}
\widetilde T_{2^s,n}=\frac{Q_{2^s}}{Q_{2^{|n|}}},\qquad 0\le s\le |n|,
\end{equation*}
and therefore \eqref{Q-cond} is equivalent to the uniform boundedness of the
corresponding $H_1(\mathbb I)\to L_1(\mathbb I)$ norms. The pointwise
convergence then follows from Theorem~\ref{divergence}\,\textup{(1)}.

Assume now that \eqref{div2} holds. With $a_s:=Q_{2^s}$ and
$A_m:=\sum_{s=0}^m a_s$, condition \eqref{div2} is exactly
\begin{equation*}
\sup_m\frac{A_m}{a_m}=\infty .
\end{equation*}
By Proposition~\ref{prop2}, there are $c>0$, integers $n_j\uparrow\infty$,
and $\gamma_j\uparrow\infty$ such that
\begin{equation*}
Q_{2^{n_j-\gamma_j}}\ge c Q_{2^{n_j}}.
\end{equation*}
Consequently, for the weights in \eqref{eq:OmegaQ},
\begin{equation*}
\Omega_{n_j-\gamma_j}(2^{n_j})=\frac{Q_{2^{n_j-\gamma_j}}}{Q_{2^{n_j}}}
\ge c.
\end{equation*}
This supplies the subsequential divergence blocks required in the proof of
Theorem~\ref{divergence}. Applying the construction there to the dyadic
subsequence $2^{n_j}$ yields an $f\in L_1(\mathbb I)$ for which the Nörlund
means $\mathcal T_n^{\mathbb Q}(f;x)$ diverge at every point
$x\in\mathbb I$. This proves part~\textup{(b1)}.
\end{proof}

\section{Appendix A: Background Material}

Let $\mathbb{P}:=\{1,2,\dots \}$, $\mathbb{N}:=\mathbb{P}\cup \{0\},\mathbb{I%
}:=[0,1)$.

We denote by$L_{1}\left( \mathbb{I}\right) $ the space of all (Lebesgue)
measurable functions $f:\mathbb{I}\rightarrow \mathbb{R}$ such that%
\begin{equation*}
\left\Vert f\right\Vert _{1}:=\int\limits_{\mathbb{I}}\left\vert
f\right\vert <\infty .
\end{equation*}%
The weak $L_{1}\left( \mathbb{I}\right) $ space, denoted by $L_{1,\infty
}\left( \mathbb{I}\right) ,$ consists of all measurable functions $f:\mathbb{%
I}\rightarrow \mathbb{R}$ for which%
\begin{equation*}
\left\Vert f\right\Vert _{1,\infty }:=\sup \lambda \left\vert \left\{
\left\vert f\right\vert >\lambda \right\} \right\vert <\infty .
\end{equation*}

Let $\{w_{n}:n\in \mathbb{N}\}$ denote the Walsh-Paley orthonormal system on 
$\mathbb{I}:=[0,1)$, and let $f\in L_{1}(\mathbb{I})$. We denote by $%
S_{n}(f;x)$ the $n$th partial sum of the Walsh--Fourier series of $f$, 
\begin{equation*}
S_{n}(f;x):=\sum_{k=0}^{n-1}\widehat{f}(k)\,w_{k}(x),\qquad n\in \mathbb{N}.
\end{equation*}

The operator $\mathcal{T}_{n}^{\mathbb{T}}$ can be written as a convolution
operator, 
\begin{equation*}
\mathcal{T}_{n}^{\mathbb{T}}(f;x) = (f \ast V_{n}^{\mathbb{T}})(x),
\end{equation*}
where 
\begin{equation*}
V_{n}^{\mathbb{T}}(u) := \sum_{k=1}^{n} t_{k,n} D_{k}(u), \qquad D_{k}(x) :=
\sum_{j=0}^{k-1} w_{j}(x)
\end{equation*}
is the $k$th Walsh--Dirichlet kernel. The corresponding weighted Lebesgue
function is defined by 
\begin{equation*}
\mathcal{L}_{n}^{\mathbb{T}} := \|V_{n}^{\mathbb{T}}\|_{1}.
\end{equation*}

Let $(f_{n})_{n\geq 0}$ be a martingale with respect to the dyadic
filtration $\{\mathcal{F}_{n}\}_{n\geq 0}$ on $\mathbb{I}$, where each $%
\mathcal{F}_{n}$ is generated by the dyadic intervals of length $2^{-n}$. We
recall the standard dyadic maximal function and square function: 
\begin{equation*}
S(f)(x):=\Biggl(\sum_{n=1}^{\infty }|f_{n}(x)-f_{n-1}(x)|^{2}\Biggr)^{1/2}.
\end{equation*}

The \emph{dyadic Hardy space} $H_{1}\left( \mathbb{I}\right) $ is defined by 
\begin{equation*}
H_{1}\left( \mathbb{I}\right) :=\{f:S(f)\in L_{1}(\mathbb{I})\},\qquad \Vert
f\Vert _{H_{1}\left( \mathbb{I}\right) }:=\Vert S(f)\Vert _{1}.
\end{equation*}

We next record some notation related to the binary expansion of integers and
to the matrix $\mathbb{T}$. For $n\in \mathbb{N}$ we write its dyadic
expansion as 
\begin{equation}
n=\sum_{j=0}^{\infty }\varepsilon _{j}(n)\,2^{j},\qquad \varepsilon
_{j}(n)\in \{0,1\}.  \label{rep2}
\end{equation}%
Set%
\begin{equation*}
\left\vert n\right\vert :=\max \left\{ s:\varepsilon _{s}\left( n\right)
=1\right\} .
\end{equation*}%
Then (\ref{rep2}) can be rewritten as%
\begin{equation*}
n=\sum_{j=0}^{\left\vert n\right\vert }\varepsilon _{j}(n)\,2^{j}.
\end{equation*}

For $s\in \mathbb{N}$ we set 
\begin{equation*}
n^{(s)}:=\sum_{j=s}^{\infty }\varepsilon _{j}(n)\,2^{j},\qquad
n(s):=\sum_{j=0}^{s}\varepsilon _{j}(n)\,2^{j}.
\end{equation*}%
Furthermore, for $0\leq m\leq n$ we define 
\begin{equation*}
T_{n}^{(m)}:=\sum_{l=m}^{n}t_{l,n}.
\end{equation*}

A useful decomposition of the kernel $V_{n}^{\mathbb{T}}$ was obtained in~%
\cite{GogNagyMathematics}. It can be written in the form 
\begin{equation*}
V_{n}^{\mathbb{T}}=\sum_{s=0}^{|n|}\varepsilon
_{s}(n)\,w_{n^{(s+1)}}\,\sum_{k=0}^{2^{s}-1}T_{n}^{\bigl(k+1+n^{(s+1)}\bigr)%
}\,w_{k},\qquad n\in \mathbb{P}.
\end{equation*}

Suppose that $\varepsilon _{s}(n)=1$ for some $s\in \{0,1,\dots ,|n|\}$.
Using the identity $2^{s}-1-k=(2^{s}-1)\oplus k$ (where $\oplus $ denotes
bitwise addition modulo $2$) and the relation $w_{n\oplus m}=w_{n}w_{m}$ (see \cite{SWS ,GES}),
and applying Abel's transformation, we obtain 
\begin{align*}
& \sum_{k=0}^{2^{s}-1}T_{n}^{\bigl(n^{(s)}-(2^{s}-1-k)\bigr)}\,w_{k} \\
& =\sum_{k=0}^{2^{s}-1}T_{n}^{\bigl(n^{(s)}-k\bigr)}\,w_{(2^{s}-1)\oplus k}
\\
& =w_{2^{s}-1}\sum_{k=1}^{2^{s}}T_{n}^{\bigl(n^{(s)}-k+1\bigr)}\,w_{k-1} \\
& =w_{2^{s}-1}\sum_{k=1}^{2^{s}-1}\Bigl(T_{n}^{\bigl(n^{(s)}-k+1\bigr)%
}-T_{n}^{\bigl(n^{(s)}-k\bigr)}\Bigr)D_{k}+w_{2^{s}-1}T_{n}^{\bigl(%
n^{(s+1)}+1\bigr)}D_{2^{s}} \\
&
=-\,w_{2^{s}-1}\sum_{k=1}^{2^{s}-1}t_{n^{(s)}-k,n}\,D_{k}+w_{2^{s}-1}T_{n}^{%
\bigl(n^{(s+1)}+1\bigr)}D_{2^{s}} \\
& =-\,w_{2^{s}-1}\sum_{k=1}^{2^{s}-2}\bigl(t_{n^{(s)}-k,n}-t_{n^{(s)}-k-1,n}%
\bigr)\,kK_{k} \\
& \quad
-\,w_{2^{s}-1}\sum_{k=1}^{2^{s}-1}t_{n^{(s+1)}+1,n}%
\,(2^{s}-1)K_{2^{s}-1}+w_{2^{s}-1}T_{n}^{\bigl(n^{(s+1)}+1\bigr)}D_{2^{s}},
\end{align*}%
where $K_{k}$ denotes the $k$th Walsh--Fejér kernel and $D_{k}$ the
Walsh--Dirichlet kernel.

Since $w_{n}w_{n^{(s+1)}}=w_{n(s)}$ and $D_{2^{s}}%
\,w_{2^{s}-1}w_{n(s)}=w_{2^{s}}D_{2^{s}}$, we arrive at the decomposition 
\begin{align}
w_{n}V_{n}^{\mathbb{T}}& =\underbrace{-\sum_{s=0}^{|n|}\varepsilon
_{s}(n)\,w_{n(s)}w_{2^{s}-1}\sum_{k=1}^{2^{s}-2}\bigl(%
t_{n^{(s)}-k,n}-t_{n^{(s)}-k-1,n}\bigr)\,kK_{k}}_{=:V_{n,1}^{\mathbb{T}}}
\label{V=V1+V2+V3} \\
& \quad \underbrace{-\sum_{s=0}^{|n|}\varepsilon
_{s}(n)\,w_{n(s)}w_{2^{s}-1}\,t_{n^{(s+1)}+1,n}\,(2^{s}-1)K_{2^{s}-1}}%
_{=:V_{n,2}^{\mathbb{T}}}  \notag \\
& \quad +\underbrace{\sum_{s=0}^{|n|}\varepsilon _{s}(n)\,T_{n}^{\bigl(%
n^{(s+1)}+1\bigr)}\,w_{2^{s}}D_{2^{s}}}_{=:V_{n,3}^{\mathbb{T}}}.  \notag
\end{align}

Thus the kernel $w_{n}V_{n}^{\mathbb{T}}$ decomposes naturally into three
components $V_{n,1}^{\mathbb{T}}$, $V_{n,2}^{\mathbb{T}}$ and $V_{n,3}^{%
\mathbb{T}}$, which will play distinct roles in the proof of Theorem \ref{divergence}.

We introduce the quantities 
\begin{equation*}
\widetilde{T}_{n(s),n}:=\sum_{l=0}^{n(s)-1}t_{n-l,n}.
\end{equation*}%
It is straightforward to check that 
\begin{equation*}
T_{n}^{\bigl(n^{(s+1)}+1\bigr)}=\sum_{l=n^{(s+1)}+1}^{n}t_{l,n}=%
\sum_{l=0}^{n(s)-1}t_{n-l,n}=:\widetilde{T}_{n(s),n}.
\end{equation*}%
When $\varepsilon _{s}\left( n\right) =1$ then 
\begin{equation*}
\widetilde{T}_{2^{s},n}\leq \widetilde{T}_{n(s),n}\leq \widetilde{T}%
_{2^{s+1},n}\leq 2\widetilde{T}_{2^{s},n}.
\end{equation*}%
We define%
\begin{equation*}
\Omega _{s}\left( n\right) :=\widetilde{T}_{2^{s},n},\text{ \ \ \ \ }0\leq
s\leq |n|.
\end{equation*}

We assume that the following limit exists: 
\begin{equation}
\lim_{(s,n)_{\Delta }\rightarrow \infty }\frac{\Omega _{s}\left( n\right) }{%
\Omega _{s-1}\left( n\right) }=\lim_{(s,n)_{\Delta }\rightarrow \infty }%
\frac{\widetilde{T}_{2^{s},n}}{\widetilde{T}_{2^{s-1},n}}=L\geq 1.
\label{limexist}
\end{equation}

It is easy to see that%
\begin{equation*}
P_{n}\left( \boldsymbol{\Omega }\right) =\sum\limits_{s=0}^{|n|}\varepsilon
_{s}\left( n\right) \widetilde{T}_{2^{s},n}w_{2^{s}}D_{2^{s}}
\end{equation*}%
and the martingale transform (\ref{conv}) can be written in the form%
\begin{equation}
M_{n}\left( \Omega \right) f=w_{n}\left( f\ast \left( w_{n}P_{n}\left( 
\boldsymbol{\Omega }\right) \right) \right) =\left( fw_{n}\right) \ast
P_{n}\left( \boldsymbol{\Omega }\right) .  \label{M1}
\end{equation}

From (\ref{V=V1+V2+V3}) we can write%
\begin{eqnarray}
f\ast V_{n}^{\mathbb{T}} &=&w_{n}\left( fw_{n}\right) \ast \left(
w_{n}V_{n}^{\mathbb{T}}\right) =w_{n}\left( fw_{n}\right) \ast V_{n,1}^{%
\mathbb{T}}  \label{M2} \\
&&+w_{n}\left( fw_{n}\right) \ast V_{n,2}^{\mathbb{T}}+w_{n}\left(
fw_{n}\right) \ast P_{n}\left( \boldsymbol{\Omega }\right) .  \notag
\end{eqnarray}

Throughout the paper we work on the dyadic group $\mathbb{I} := [0,1)$
equipped with the dyadic expansion 
\begin{equation*}
x = \sum_{j=0}^{\infty} x_{j} 2^{-(j+1)}, \qquad x_{j} \in \{0,1\}.
\end{equation*}

For each integer $m \ge 1$ and any choice of digits $t_{0},\dots,t_{m-1} \in
\{0,1\}$, we define the dyadic interval 
\begin{equation*}
I_{m}(t_{0},\dots,t_{m-1}) := \bigl\{ x \in \mathbb{I} : x_{j} = t_{j} \ 
\text{for } j = 0,\dots,m-1 \bigr\}.
\end{equation*}
Equivalently, 
\begin{equation*}
I_{m}(t_{0},\dots,t_{m-1}) = \Biggl[ \sum_{j=0}^{m-1} t_{j} 2^{-(j+1)}, \
\sum_{j=0}^{m-1} t_{j} 2^{-(j+1)} + 2^{-m} \Biggr).
\end{equation*}

In particular, in the notation 
\begin{equation*}
I_{a}\bigl(t_{0},\dots ,t_{a-\gamma (a)-1},t_{a-2\gamma (a)-1},\dots
,t_{a-\gamma (a)-1}\bigr)
\end{equation*}%
we mean the dyadic interval of rank $a$ whose dyadic digits coincide with
the specified values at the corresponding positions; the remaining digits
are irrelevant for the expressions in which this notation appears.

\section*{Conflicts of Interest}

The authors declare that they have no conflicts of interest.

\section*{Data Availability}

Not applicable.

\end{document}